\documentclass[reqno,10pt]{amsart}
\usepackage[utf8]{inputenc}
\usepackage[T1]{fontenc}
\usepackage[english]{babel}
\usepackage{amssymb}
\usepackage[foot]{amsaddr}
\usepackage{a4wide}

\usepackage{graphicx}
\usepackage{wrapfig}
\usepackage{enumerate}
\usepackage[usenames,dvipsnames]{color}

\usepackage{mathtools}
\mathtoolsset{showonlyrefs,showmanualtags}

\usepackage{esint}

\usepackage{pstricks,pst-plot,pst-math} 
\usepackage{pstricks-add}

\allowdisplaybreaks

\newtheorem{thm}{Theorem}[section]

\newtheorem*{thmxx}{Main Theorem}
\newtheorem*{thmxx*}{Main Theorem*}
\newtheorem*{thmxx**}{Main Theorem**}

\newtheorem{defn}[thm]{Definition}
\newtheorem{lem}[thm]{Lemma}
\newtheorem{prop}[thm]{Proposition}

\newtheorem*{conjecture*}{Conjecture}

\newtheorem{innercustomgeneric}{\customgenericname}
\providecommand{\customgenericname}{}
\newcommand{\newcustomtheorem}[2]{%
	\newenvironment{#1}[1]
	{%
		\renewcommand\customgenericname{#2}%
		\renewcommand\theinnercustomgeneric{##1}%
		\innercustomgeneric
	}
	{\endinnercustomgeneric}
}
\newcustomtheorem{customthm}{Theorem}

\theoremstyle{definition}

\theoremstyle{remark}
\newtheorem{rem}[thm]{Remark}

\numberwithin{equation}{section}

\newcommand{\DeclareAutoPairedDelimiter}[3]{%
	\expandafter\DeclarePairedDelimiter\csname Auto\string#1\endcsname{#2}{#3}%
	\begingroup\edef\x{\endgroup
		\noexpand\DeclareRobustCommand{\noexpand#1}{%
			\expandafter\noexpand\csname Auto\string#1\endcsname*}}%
	\x}

\DeclareAutoPairedDelimiter{\abs}{\lvert}{\rvert}
\DeclareAutoPairedDelimiter{\norm}{\lVert}{\rVert}
\DeclareAutoPairedDelimiter{\bra}{(}{ )}
\DeclareAutoPairedDelimiter{\pra}{[}{]}
\DeclareAutoPairedDelimiter{\set}{\{}{\}}
\DeclareAutoPairedDelimiter{\skp}{\langle}{\rangle}

\DeclareMathAlphabet{\mathup}{OT1}{\familydefault}{m}{n}
\newcommand{\dx}[1]{\mathop{}\!\mathup{d} #1}
\newcommand{\dH}[1]{\mathop{}\!\mathup{d} \mathcal{H}^{#1}}

\DeclareMathOperator{\loc}{loc}

\DeclareMathOperator{\tr}{tr}

\DeclareMathOperator{\Lip}{Lip}

\newcommand{\N}{\mathbb{N}}
\newcommand{\R}{\mathbb{R}}

\newcommand{\cC}{\ensuremath{\mathcal C}}

\newcommand{\cH}{\ensuremath{\mathcal H}}

\newcommand{\cN}{\ensuremath{\mathcal N}}

\newcommand{\cP}{\ensuremath{\mathcal P}}

\newcommand{\tP}{\ensuremath{\tilde P}}
\newcommand{\coeff}{\ensuremath{c}}

\usepackage[colorinlistoftodos,prependcaption,textsize=tiny,linecolor=red,backgroundcolor=red!25,bordercolor=red]{todonotes}

\newcommand{\news}[1]{{\textcolor{magenta}{#1}}}

\usepackage[pdfusetitle]{hyperref}
\hypersetup{
    colorlinks=true,       
    linkcolor=black,          
    citecolor=black,        
    filecolor=black,      
    urlcolor=black           
}

\title[On global solutions of the obstacle problem]{On global solutions of the obstacle problem}

\author{Simon Eberle$^1$}
\address{$^1$Basque Center for Applied Mathematics}
\email{seberle@bcamath.org}
\author{Henrik Shahgholian$^2$}
\address{$^2$Department of Mathematics, KTH Royal Institute of Technology, Sweden}
\email{henriksh@kth.se}
\author{Georg S. Weiss$^3$}
\address{$^3$Faculty of Mathematics, University of Duisburg-Essen, Germany}
\email{georg.weiss@uni-due.de}

\let\rho\varrho
\let\phi\varphi
\let\epsilon\varepsilon
\let\emptyset\varnothing

\thanks{H. Shahgholian was supported in part by the Swedish Research Council.}
\thanks{S. Eberle and G. Weiss were both during the first months of the project guests of the Hausdorff Research Institute for Mathematics of the university of Bonn.}

\begin{document}
	
	\begin{abstract}
Assuming a lower bound on the dimension, we prove 
a long standing conjecture concerning the 
classification of global solutions of the obstacle problem with \emph{unbounded} coincidence sets.
	\end{abstract}

		\maketitle

\tableofcontents

\section{Introduction}

Many important problems in science, finance and engineering can be modeled by PDEs that have a-priori unknown interfaces. Such problems are called \emph{free boundary problems} and they have been a major area of research in the field of PDEs for at least half a century.
The \emph{obstacle problem} arguably is the most extensively   studied free boundary problem. It may be derived from a simple model for spanning an elastic membrane over some given (concave) obstacle. Alternatively it can be derived from a certain setting in the Stefan problem, the simplest model for the melting of ice, or from the Hele-Shaw problem. 
The obstacle problem may be expressed in a single nonlinear partial differential equation 
\begin{align}\label{eq:obstacle-local}
\Delta u=\coeff(x)\chi_{\{u>0\}}, ~ u \geq 0 , ~ \text{ where } \coeff \in \Lip(\R^N, [\coeff_0, \infty)), \coeff_0 >0,
\end{align}
and  the set $\{ u>0\}$ models 
the phase where the membrane does not touch the obstacle
respectively  the water phase in the stationary Stefan problem. 
The set $\{ u=0\}$ is called {\em coincidence set} and the interface $\partial\{ u>0\}$ is called the {\em free boundary}. \\
From the point of view of mathematics, the most challenging question in free boundary problems is the \emph{structure and the regularity of the free boundary}. The development of contemporary regularity theory for free boundaries started with the seminal work of L. Caffarelli \cite{Caffarelli_regularity_free_boundary_higher_dimensions_Acta77} in the late 70s, and since then has been a very active field of research. \\ As the obstacle problem has been extensively studied over the last $4$ decades  of the questions originally posed only the hard problems have remained unsolved. Among them are the fine structure of the singular part of the free boundary and the behavior of the regular part of the free boundary close to singularities. Towards answering the first question there have been many impressive results in recent years (cf. A. Figalli, J. Serra \cite{Figalli-Serra-2020-Inventiones}, A. Figalli, J. Serra and X. Ros-Oton \cite{Figalli-Ros-Oton-Serra-Generic}, M. Colombo, L. Spolaor and B. Velichkov \cite{ColomboSpolaorVelnichkov18}, O. Savin and H. Yu \cite{SavinYu2020regularity}). While the singular set has thus been extensively studied there has -- to the authors' best knowledge -- been  no result on the behavior of the \emph{regular part of the free boundary close to singularities}.

\subsection{Global solutions of the obstacle problem}
In the study of the behavior of the 
regular part of the free boundary close to singularities,
the structure of global/entire solutions of the obstacle problem
\begin{equation}\label{eq:obstacle-global}
\Delta u(x)= \chi_{\{u>0\}}, \qquad u \geq 0, \qquad \hbox{in }\R^{N},
\end{equation}
plays a vital role.
Note that the study of (compact) coincidence sets of such global solutions
goes  back to Isaac Newton (see subsection \ref{pot-theory}), who stated his famous \emph{no gravity in the cavity theorem} in 1678. \\
The first partial classification of global solutions with compact coincidence sets has been achieved almost 90 years ago in 1931 by  P. Dives \cite{Dive} who showed in \emph{three dimensions}  ---in the language of potential theory--- that if $\set {u=0}$ has non-empty interior and is \emph{bounded} then it is an ellipsoid. Many years later H. Lewy in 1979 gave a new proof in \cite{Lewy79}.
In 1981, M. Sakai gave a full classification of global solutions in \emph{two dimensions} using complex analysis (cf. \cite{Sakai_Null_quadrature_domains}). The higher dimensional analogue to Dive's result, i.e. if $\set {u=0}$ is bounded and has non-empty interior then it is an ellipsoid, was proved shortly after in two steps. First E. DiBenedetto and A. Friedman proved the result in 1986 under the additional assumption that $\set {u=0}$ is symmetric with respect to $\set {x_j =0}$ for all $j \in \set {1, \dots, N}$ (cf. \cite{DiBenedettoFriedman}). In the same year A. Friedman and M. Sakai \cite{FriedmanSakai} removed this extra symmetry assumption.
In \cite{ellipsoid} two of the authors gave a very concise proof of the characterization of compact coincidence sets. \\
It is noteworthy that \cite{FriedmanSakai} has a beautiful application to Eshelby's inclusion problem (\cite{Kang-Milton-2008,Liu-2008,KANG_Eshelby_review}). For more details we refer the interested reader to section \ref{sec:Eshelby}.

While global solutions with compact coincidence sets have thus been completely classified, the structure of solutions with \emph{unbounded} coincidence sets has been largely open and is related to the following 30 year-old conjecture:
\begin{center}
	{\it The coincidence set of each global solution of the obstacle problem is a half-space, an ellipsoid, a paraboloid or a cylinder with 
		ellipsoid or paraboloid as base.}
\end{center}
This conjecture has first been raised by one of the authors in \cite[conjecture on p. 10]{Shahgholian92_conjecture} and was later reiterated  in \cite[Conjecture 4.5]{KarpMargulis_bounded_sources}. 

We prove the conjecture under a technical assumption:

\begin{thmxx}\label{thm:MainTheorem_Intro_I}
	Let  $u$ be a   solution of   \eqref{eq:obstacle-global} 
	such that the coincidence set $\{ u=0\}$ has non-empty interior
	and such that,
	setting 
	\begin{align}
	p(x) := \lim \limits_{\rho \to \infty} \frac{u(\rho x)}{\rho^2}\quad \text{ in } L^\infty(\partial B_1) \quad \text{and} \quad \cN(p) := \{p=0\},
	\end{align}
	$\dim \cN(p) \leq N-5$.
	\\
	Then the coincidence set is an ellipsoid,
	a paraboloid or a cylinder with 
	ellipsoid or paraboloid as base.
\end{thmxx}
Note that the bound on the dimension of $\cN(p)$ does not only come from the requirement of the Newton potential to be well defined on the coincidence set $\cC$ but also from the fact that for $\dim \cN(p) > N-5$ the asymptotic behavior of $u-p$ is more involved (for more details see the end of section \ref{sec:structure_of_proofs}).

In a forthcoming paper we will apply the Main Theorem
to the analysis of the behavior of the 
regular part of the free boundary close to singularities.
The existence of paraboloid solutions with precise asymptotic behavior 
Section \ref{section:existence_of_paraboloid_solutions}
also implies the existence of paraboloid traveling wave solutions in
the Hele-Shaw problem with precise asymptotic behavior at infinity.
Although this is in no way the focus of our paper,
we briefly comment on those traveling waves in the Appendix \ref{sec:Traveling_waves_Hele-Shaw}.
Moreover we will describe the potential theoretic aspect 
of global solutions of the obstacle problem in some detail
in the Appendix \ref{sec:Potential_Theory_and_obstacle_problem}.  
\subsection{Structure of the proofs} \label{sec:structure_of_proofs}

The proof of the Main Theorem 
uses precise estimates on the asymptotic behavior  of the solution $u$ at infinity to allow  a comparison with  a paraboloid solution with matched asymptotics.
In section \ref{early-results} we first reduce the problem to its non-cylindrical (cf. Definition \ref{def:cylindrical}) dimensions and to the case of \emph{unbounded} coincidence set. We show that in this reduced setting it holds that $\operatorname{dim}(\cN(p))=1$ so that the assumption in our Main Theorem
becomes $N \geq 6$.
 In section \ref{section:first_frequency_estimate} we give a first frequency estimate and show that a rescaled version of $u-p$ converges to an affine linear function $\ell$. In section \ref{section:first_estimate_on_the_coincidence_set} we infer from this preliminary analysis that the coincidence set $\cC$ is asymptotically contained in a set slightly larger than a paraboloid. Note that sections 4-5 are valid in any dimension
$\ge 3$.
As a consequence the Newtonian potential of $\cC$ is well defined in higher dimenions (see section \ref{section:Newton_potential_expansion_of_u})
so that we may expand $u$ into a quadratic polynomial and the Newtonian potential. In section \ref{section:existence_of_paraboloid_solutions}
we construct a paraboloid solution matching the quadratic and linear asymptotic behavior of $u$ at infinity.
We conclude the proof of the Main Theorem  
in section \ref{section:comparison}. In that section we first prove that the Newtonian potential converges uniformly to zero outside a set slightly larger than the paraboloid, as $\abs{x} \to \infty$ so that the Newtonian potential expansion of the solution into a quadratic polynomial is rigorous outside that set.
Finally we use a comparison principle with mismatched data on some boundary part to show that $u$ lies below some translated version of the paraboloid solution. A sliding argument concludes the proof of the Main Theorem.  

In dimension $N=4$ and $N=5$ the difference
$u-p-\ell$ is no longer bounded which makes more intricate methods necessary, subject of future research.

In dimension $N=3$ the difference $u-p$ is on the ball $B_R$ of order
$R \log R$ and thus has superlinear growth, so $N=3$ is a truely critical case 
---where even matching the next order (below quadratic) in the asymptotic expansion at infinity poses a formidable difficulty---  
which requires 
substantially different methods.

\section*{Acknowledgments}
The authors thank the Hausdorff Research Institute for Mathematics of the university of Bonn for its hospitality and Herbert Koch for discussions.

\section{Notation}\label{sec:notation}

We shall now  clarify the notation used in the introduction and make some assumptions that will make notational complexity as low as possible in the rest of the paper.
\\
Throughout this work $\R^N$ will be equipped with the Euclidean inner product $x \cdot y$ and the induced norm $\abs{x}$. Due to the nature of the problem we will often write $x \in \R^N$ as $x=(x', x'') \in \R^{N-n} \times \R^n$ for an $n \in \N$, $N \geq n+1$. $B_r(x)$ will be the open $N$-dimensional ball of center $x$ and radius $r$. $B_r'(x')$ will be the open $(N-n)$- dimensional ball of center $x' \in \R^{N-n}$ and radius $r$. Whenever the center is omitted it is assumed to be $0$.
\\
When considering a set $A$, $\chi_A$ shall denote the characteristic function of $A$. $\cH^{N-1}$ is the $(N-1)$-dimensional Hausdorff measure.
\\
Finally we call special solutions of the form
$\max \{x\cdot e,0\}^2/2$ half-space solutions; here $e\in \partial B_1$ is a fixed vector.\\
When $M \in \R^{N \times N}$ is a matrix, we mean by $\tr(M):= \sum_{j=1}^N M_{jj}$ its trace.

\begin{defn}[Coincidence set]\label{def:coincidence_set}
	\mbox{} \\
	For solutions $u$ of the obstacle problem \eqref{eq:obstacle-global}, we define the \emph{coincidence set} $\cC$ to be
\begin{align}
\cC := \set {u=0}.
\end{align}
\end{defn}

\begin{rem} 	It is known that the coincidence  $\cC$ of a \emph{global} solution of the obstacle problem is \emph{convex} (see e.g. \cite[Theorem 5.1]{PetrosyanShahgholianUraltseva_book}).
\end{rem}

\begin{defn}[Cylindrical and non-cylindrical solutions] \label{def:cylindrical}
	\mbox{}\\
We call a global solution $u$ of the obstacle problem  \eqref{eq:obstacle-global} \emph{cylindrical} (in the direction $e$) if for an $e \in \partial B_1$
\begin{align}
	\nabla u \cdot e \equiv 0 \quad \text{ in } \R^N
\end{align}
and we call $u$ \emph{non-cylindrical} if it is not cylindrical.
\end{defn}

\begin{defn}[Ellipsoids and Paraboloids] \label{def:ellipsoid_and_paraboloid}
	\mbox{} \\
We call a set $E \subset \R^N$ \emph{ellipsoid} if after translation and rotation
\begin{align}
	E = \set {x \in \R^N : \sum \limits_{j = 1}^N \frac{x_j^2}{a_j^2} \leq 1   }
\end{align}
for some $a \in (0,\infty)^N$.  We call a set $P \subset \R^N$ a \emph{paraboloid}, 
if after translation and rotation  $P $ can be represented as 
$$
	P = \set { (x',x_N) \in \R^N : x' \in \sqrt{x_N} E'   } ,
$$	  
where  $E' $  is an $(N-1)$-dimensional ellipsoid.
\end{defn}
\begin{defn}[Newtonian potential]\label{NP}
	\mbox{} \\
	Let $N \geq 3$, let $M \subset \R^N$ be a measurable set, and let $\alpha_N := \frac{1}{N(N-2) \abs {B_1}}$. 
	We call 
	\begin{align}
		V_M(x) := \alpha_N \int \limits_M \frac{1}{  \abs {x-y}^{N-2}   } \dx{y} \quad \in [0,+\infty] 
	\end{align} 	the \emph{Newtonian potential of $M$}.
        We say that the Newtonian potential is well-defined if $V_M(x)<+\infty$ for every $x\in \R^N$, and satisfies 
        $$
        \Delta V_M = -\chi_{M}, \quad \hbox{in } \R^N.
        $$
\end{defn}


\section{Known results and  reduction to the paraboloid case}\label{early-results}
\subsection{Known results}
In this section we shall recall some known results concerning classification of global solutions of the obstacle problem. 
We gather them in the following proposition.
\begin{prop}[Known Properties] \label{prop:known_results}
	Let $u$ be a global  solution of the obstacle problem \eqref{eq:obstacle-global}. Then:
	
	\begin{enumerate}[(i)]	
		\item The second derivatives are globally bounded, i.e. there is $C<+\infty$ such that 
		\begin{equation}\label{bounded_second_derivative}
		\norm {D^2u}_{L^\infty(\R^N)} \leq C.
		\end{equation}
		
		\item \label{prop:known_results_item_2}If the coincidence set $\cC$ contains two sequences $(x^j)_{j \in \N}$, $(y^k)_{k \in \N} \subset \cC$, such that  $\abs{x^j}\to \infty$, $\abs{y^k} \to \infty$ as $j \to \infty$ and $k \to \infty$ and 
		$ \tilde x^j:=x^{j}/|x^j|, \tilde y^k:=y^{k}/|y^k|$  converge to two independent vectors $x^0, y^0 \in \partial B_1$ on the unit sphere, then the global solution $u$ is cylindrical. 
		
		\item \label{prop:known_results_item_3}  If the sequences above have the property that the limit vectors $x^0, y^0$ satisfy $x^0 = -y^0$, then the global solution $u$ is cylindrical in the $x^0$--direction. 
		
		\item \label{prop:known_results_item_4} If the coincidence set is unbounded in the $e^N$-direction and not cylindrical in any direction $e \perp e^N$,  then the  
		\emph{blow-down  is independent of the $e^N$-direction only}. More precisely
		\begin{align}
		u_r(x)	:= \frac{u(rx)}{r^2} \to x'^T Q x' =:p(x') ~\text{ in } C^{1,\alpha}_{\operatorname{loc}}\cap
		W^{2,p}_{\operatorname{loc}}\text{ as } r \to \infty,
		\end{align}
		where $x=(x',x_N)$, $Q  \in \R^{N-1 \times N-1}$ is positive definite, symmetric and $\tr(Q) =  \frac{1}{2}$. 
	\end{enumerate}
	
\end{prop}

\begin{proof}
	\mbox{}\\
\begin{enumerate}[(i)]
\item This is a direct consequence of \cite[Theorem 2.1]{PetrosyanShahgholianUraltseva_book}.
\item This follows from \cite[proof of Theorem II, case 3]{CaffarelliKarpShahgolian_Annals_2000}.
\item This follows by directional monotonicity  and convexity of $u$
(cf. \cite[proof of Theorem II, Case 3]{CaffarelliKarpShahgolian_Annals_2000}). 
\item For the uniqueness of the blow-down (i.e. independence of sequences) see Lemma \ref{lem:uniqueness_of_blow-downs}. For the strong $W^{2,p}_{\loc}(\R^N)$-convergence we refer to \cite[Proposiion 3.17 (v)]{PetrosyanShahgholianUraltseva_book}. For the fact that the blow-down limit must be a polynomial solution we refer to \cite[Proposition 5.3]{PetrosyanShahgholianUraltseva_book} and the fact that a solution that has the half-space solution as blow-down must itself be the half-space solution (cf. \cite[Proof of Theorem II Case 2]{CaffarelliKarpShahgolian_Annals_2000}). \\
Finally, from the fact that the blow-down is a polynomial solution and the assumption that the coincidence set $\cC$ vanishes in $e^N$-direction, we infer that the polynomial must be independent of $x_N$. It remains to show that the polynomial does not vanish in any other direction. Assume towards a contradiction that $Q$ vanishes in some direction $\tilde{e} \in \partial B_1 \setminus \{-e^N,e^N\}$. Then we conclude, applying the ACF monotonicity function to $\partial_{\tilde{e}}u$ (cf. \cite[Proof of Theorem II Case 2]{CaffarelliKarpShahgolian_Annals_2000}), that $u$ is monotone in the $\tilde{e}$-direction. But this implies that the coincidence set $\cC$ must be unbounded in this direction. Then \eqref{prop:known_results_item_2} implies that $u$ is cylindrical, contradicting our assumption.
\end{enumerate}
\end{proof}

\subsection{Reduction} \label{sec:reduction}

	For the application to the study of the behavior of the 
regular part of the free boundary close to singularities we actually need a slightly stronger version of the Main Theorem that precisely relates the sectional ellipsoids of paraboloidal coincidence sets with the asymptotic behavior (blow-down) of a global solution $u$. Instead of the Main Theorem we will prove the extended result stated below.

\begin{thmxx*}\label{thm:MainTheoremI}
	Let  $u $ be a   solution of  \eqref{eq:obstacle-global} that is cylindrical (cf. Definition \ref{def:cylindrical}) in precisely $k$ independent directions. If $N-k \geq 6$ and the coincidence set $\set {u=0}$ has non-empty interior, then the restriction of the coincidence set to the non-cylindrical directions is either an ellipsoid or a paraboloid
	(in the sense of Definition \ref{def:ellipsoid_and_paraboloid}). \\
		The ellipsoids as well as the cross sections of the  paraboloid are given as the (up to scaling and rigid motion) \emph{unique} ellipsoids that are given as the coincidence set $\{v'=0\}$ of  a solution of the ($N-k-1$)-dimensional obstacle problem with \emph{the same blow-down}, i.e. $v': \R^{N-k-1} \to [0,\infty)$ solves
		\begin{align}
			\Delta v' = \chi_{\set {v'>0}} ~&\text{ in  } \R^{N-k-1} \quad \text{ and } \\
			\lim \limits_{\rho \to \infty} \frac{v'(\rho x')}{\rho^2} = p(x')=\lim \limits_{\rho \to \infty} \frac{u(\rho x)}{\rho^2}~&\text{ in } L^\infty(\partial B_1) .
		\end{align}	
\end{thmxx*}

	Let $u$ still be a global  solution of the obstacle problem \eqref{eq:obstacle-global}.
	Rotating and considering the restriction of the solution to all non-cylindrical directions we may assume that $u$ is non-cylindrical	
	and that $k$ (defined in Main Theorem*) 
 is $0$. 
Since bounded coincidence sets of global solutions are already known to be ellipsoids (see \cite{Dive}, \cite{DiBenedettoFriedman}), we  shall henceforth discuss only the case of unbounded coincidence sets. Therefore it is sufficient to prove the following reduced version of Main Theorem*.

\begin{thmxx**}\label{thm:Main_Theorem_of_dissertation}
	\mbox{}\\
	Let $N \geq 6$ and let $u$ be a solution of \eqref{eq:obstacle-global} that is \emph{non-cylindrical} (in the sense of Definition \ref{def:cylindrical}). If furthermore $\set {u=0}$ is \emph{unbounded} and has \emph{non-empty interior}, then $\set {u=0}$ is a paraboloid (in the sense of Definition \ref{def:ellipsoid_and_paraboloid}). \\
		The cross sections of the  paraboloid are given as the (up to scaling and translation) \emph{unique} ellipsoids that are given as the coincidence set $\{v'=0\}$ of  a solution of the ($N-1$)-dimensional obstacle problem with \emph{the same blow-down}, i.e. $v': \R^{N-1} \to [0,\infty)$ solves
\begin{align}
	\Delta v' = \chi_{\set {v'>0}} ~&\text{ in  } \R^{N-1} \quad \text{ and } \\
	 \lim \limits_{\rho \to \infty} \frac{v'(\rho x')}{\rho^2} = p(x')=\lim \limits_{\rho \to \infty} \frac{u(\rho x)}{\rho^2}~&\text{ in } L^\infty(\partial B_1) .
\end{align}	
\end{thmxx**}
Note that $u$ satisfying the assumptions of Main Theorem** 
implies that the conclusion of Proposition \ref{prop:known_results} \eqref{prop:known_results_item_4} holds.
It follows that, rotating if necessary, the convex, closed set 
\begin{align}
\cC \text{ must contain a ray in the positive } e^N\text{-direction}.
\end{align}
Moreover
\begin{align}
  \partial_{N}u \leq 0 \text{ in } \R^N
 \end{align}
see \cite[Proof of Case 2 of Theorem II]{CaffarelliKarpShahgolian_Annals_2000}. 
Hence for any free boundary point $z$, we obtain that the ray $L_z:= \{z + te^N : t \geq 0\}$ is contained in the coincidence set, i.e. $L_z \subset \cC$.  Since $u$ is non-cylindrical, the coincidence set $\cC$ cannot contain a line 
{(cf. Proposition \ref{prop:known_results}
 \eqref{prop:known_results_item_3}), and using the convexity of $\cC$,  $\lambda : = \min\{ z_N:\ z\in \cC \}  > -\infty$, and hence 
\begin{align}
\cC \subset \{ x_N \geq \lambda \} \text{ for some } \lambda \in \R.
\end{align}

By the above discussion we may after suitable change of  variables assume the solution in Main Theorem**  
to satisfy

\begin{defn} \label{def:solution}
Let $u$ be a solution of 
\begin{align}
	\Delta u = \chi_{\set {u>0} } \quad , \quad u \geq 0 \quad \text{ in } \R^N
\end{align}
(in the sense of distributions) satisfying
\begin{enumerate}[(i)]
	\item $\cC$ has non-empty interior,
	\item $\set{0} = \cC \cap \set {x_N \leq 0}$,
	\item  \label{PDE_asymptotics}
$ 
u_r(x)	:= \frac{u(rx)}{r^2} \to x'^T Q x' =:p(x') \quad \text{ in } C^{1,\alpha}_{\operatorname{loc}}\cap
W^{2,p}_{\operatorname{loc}}\text{ as } r \to \infty,
$ 	where $x=(x',x_N)$, $Q  \in \R^{N-1 \times N-1}$ is positive definite, symmetric and $\tr(Q) =  \frac{1}{2}$. \\
For later reference let us state that this means that there is $c_p>0$  such that for all $x'\in \R^{N-1}$
\begin{align} \label{eq:def_of_c_p}
p(x') \geq c_p \abs {x'}^2.
\end{align}
\end{enumerate}
\end{defn}


\section{First Frequency Estimate}  \label{section:first_frequency_estimate}

In this section we derive a first estimate of the asymptotics of the given solution $u$ as in Definition \ref{def:solution} by studying the blow-down of a normalized solution. This analysis is based on the following frequency estimate which is inspired by the monotonicity formulas
in \cite{Weiss_homogeneity_improvement}, \cite{Monneau2003} as well as the frequency formula in \cite{Figalli-Serra-2020-Inventiones}.

\begin{lem}[First Frequency estimate] \label{lem:first_frequency_estimate}
	\mbox{} \\
	Let $\tilde{v}_r$ be given as 
	\begin{align}
	\tilde{v}_r := u_r -p,
	\end{align}
where  $u_r$ is the rescaling  and $p$ the blow-down limit as introduced in Definition \ref{def:solution} \eqref{PDE_asymptotics}, and let the first frequency function be defined for all $r>0$ as 
	\begin{align}
	F_1(r) := \int \limits_{B_1} \abs {\nabla \tilde{v}_r}^2 - 2 \int \limits_{\partial B_1} \tilde{v}_r^2  \dH{N-1}  .
	\end{align}
	Then $F_1$ is monotone increasing in $r$ and $F_1(r)$ is non-positive for all $r>0$.
\end{lem}
\begin{proof}
	Note that $\tilde{v}_r$ solves $\Delta \tilde{v}_r= -\chi_{\set {u_r=0}}$ in $\R^N$.
	Observe that $F_1$ is monotone increasing as
	\begin{align}
	\frac{\dx{}}{\dx{r}} F_1(r) &= 2 \int \limits_{B_1} \nabla \tilde{v}_r \cdot \nabla \partial_r \tilde{v}_r - 4 \int \limits_{\partial B_1} \tilde{v}_r \partial_r \tilde{v}_r  \dH{N-1} \\
	&= 2 \left(\int \limits_{\partial B_1} \nabla \tilde{v}_r \cdot x ~\partial_r \tilde{v}_r  \dH{N-1} - \int \limits_{B_1} \underbrace{\Delta \tilde{v}_r \partial_r \tilde{v}_r}_{=0} - 2 \int \limits_{\partial B_1} \tilde{v}_r \partial_r \tilde{v}_r        \dH{N-1}       \right) \\
	&= 2 \int \limits_{\partial B_1} \bra { \nabla \tilde{v}_r \cdot x -2 \tilde{v}_r  } \partial_r \tilde{v}_r  \dH{N-1} = {2} r \int \limits_{\partial B_1} \bra {\partial_r \tilde{v}_r}^2  \dH{N-1} \geq 0.
	\end{align}
	By definition of $p$ we know that
	\begin{align}
	\tilde{v}_r \to 0 \quad \text{ in } C_{\operatorname{loc}}^{1,\alpha}(\R^N) \quad  \text{as } r \to \infty  .
	\end{align}
	We conclude that
	\begin{align} \label{eq:first_frequency_bound}
	\lim \limits_{r \to \infty} \bra { \int \limits_{B_1} \abs {\nabla \tilde{v}_r}^2 -2 \int \limits_{\partial B_1} \tilde{v}_r^2  \dH{N-1}  }= 0.
	\end{align}
	As $F_1$ is monotone increasing it follows that
	\begin{align}
	\int \limits_{B_1} \abs {\nabla \tilde{v}_r}^2 \leq 2 \int \limits_{\partial B_1} \tilde{v}_r^2 \dH{N-1}
	\end{align}
	 for all $r >0$ and that $F_1$ is non-positive.
\end{proof}

\subsection{The second term in the asymptotic expansion at infinity}

We already know that $u$ has quadratic growth at infinity and that its leading order asymptotics is given by $p$. In order to get information on the next order in the asymptotic expansion the usual ansatz is to normalize $u_r-p$ and pass to the limit $r \to \infty$ in the normalization. For technical reasons we will in the present section subtract a projection from this difference. Note however that as a result of Section \ref{section:Newton_potential_expansion_of_u}, we will at that stage be able to determine the limit of the normalized $u_r-p$, too. 

We define for all $r>0$
\begin{align} \label{def:v_r}
v_r := u_r -p -h_r, 
\end{align} 
where $h_r(x') :=\Pi' (u_r-p)$ with 
\begin{align} \label{def:h_r}
 \Pi' (u_r-p)	\text{ being the }L^2(\partial B_1)\text{-projection of } u_r -p \text{ onto } \cP_2' ,
\end{align} 
and
	$\cP_2'$ is the set of homogeneous harmonic polynomials of degree $2$ depending only on $x'$.
Note that $v_r$ solves
\begin{align}
	\Delta v_r = - \chi_{\set {u_r=0}} \quad \text{ in } \R^N \text{ for all } r>0.
\end{align}
Recall that for all $r>0$ we have assumed that $h_r$ is harmonic and homogeneous of degree $2$ and note that
$F_1$ is invariant with respect to perturbations by any
harmonic homogeneous polynomial $q$ of degree $2$, i.e. for all $r>0$,
\begin{align}
		F_1[\tilde{v}+q](r) &:= \int \limits_{B_1} \abs { \nabla (\tilde{v}_r +q)}^2 -2 \int \limits_{\partial B_1} \bra { \tilde{v}_r +q }^2 \dH{N-1} \\
		&=  \int \limits_{B_1} \abs { \nabla \tilde{v}_r }^2 -2 \int \limits_{\partial B_1} \bra { \tilde{v}_r  }^2 \dH{N-1} =: F_1[\tilde{v}](r) . \label{eq:frequency_is_independent_of_harmonic_quadratic_polynomials}
	\end{align}
Hence we conclude that $v_r$ satisfies the same
frequency estimate as $\tilde{v}_r$, i.e. for all $r>0$
\begin{align} \label{eq:first_frequency_estimate}
	\frac{\int_{B_1} \abs{\nabla v_r}^2 }{\int_{\partial B_1} v_r^2  \dH{N-1}} \leq 2.
\end{align}
This immediately implies a first estimate on the normalized family
\begin{align} \label{def:w_r}
w_r := \frac{v_r}{\sqrt{\int_{\partial B_1} v_r^2  \dH{N-1}}},
\end{align}
i.e. by a Poincar\'e Lemma
\begin{align}
	w_r \text{ is bounded in } W^{1,2}(B_1) \text{ uniformly in } r>0.
\end{align}
However this bound is valid only in $B_1$.
The following Lemma will give a local bound in
$\R^N$.
\begin{lem} \label{lem:w_r_in_W_1_2_loc}
	Let $w_r$ be as defined above. Then
	\begin{align}
	(w_r)_{r>1} \text{ is bounded in } W^{1,2}_{\operatorname{loc}}(\R^N).
	\end{align}
\end{lem}
\begin{proof} 
	\mbox{}\\
\textbf{Step 1.} \emph{Frequency bound on large balls and Poincaré's lemma.}\\
We claim that for each $R>0$ there is $C_1(R) < + \infty$ such that for all $r>0$ and $q \in \cP_2'$
\begin{align} \label{eq:frequency_estimate_on_large_spheres}
\int \limits_{B_R} \abs{ \nabla (\tilde{v}_r-q) }^2 \leq C_1(R) \int \limits_{\partial B_R} (\tilde{v}_r-q)^2 \dH{N-1}.
\end{align}
This is a direct consequence of Lemma \ref{lem:first_frequency_estimate} and \eqref{eq:frequency_is_independent_of_harmonic_quadratic_polynomials}. From these two we infer that for every $R,r>0$ and $q \in \cP_2'$
\begin{align}
	2 \geq \frac{\int \limits_{B_1} \abs{ \nabla ( \tilde{v}_{rR} -q )  }^2  }{ \int \limits_{\partial B_1} (\tilde{v}_{rR} -q )^2 \dH{N-1}  } = R  \frac{  \int \limits_{B_R}  \abs{ \nabla( \tilde{v}_r-q )  }^2  }{  \int \limits_{ \partial B_R } (\tilde{v}_r-q)^2 \dH{N-1}  }
\end{align}
which proves the claim. \\
We are now going to apply the following Poincaré-Lemma: There is $C_2(R) < \infty$ such that for every $g \in W^{1,2}(B_R)$:
\begin{align}
	\norm{g}^2_{L^2(B_R)} \leq C_2(R) \bra { \norm{\nabla g}^2_{L^2(B_R)} + \norm{g}^2_{L^2(\partial B_R)}   }.
\end{align}
Combining this Poincaré-Lemma with \eqref{eq:frequency_estimate_on_large_spheres} (for $q= h_r$) 
 we conclude that for each $R>1$ and every $r>0$,
\begin{align}
\norm{ w_r }^2_{W^{1,2}(B_R)} &= \frac{ \int_{B_R} \abs{ \nabla (\tilde{v}_r-h_r)  }^2 + \int_{B_R} ( \tilde{v}_r -h_r )^2   }{ \int_{\partial B_1} (\tilde{v}_r-h_r )^2 \dH{N-1}  } \\
 &\leq C_3(R) \frac{ \int_{B_R} \abs{\nabla v_r  }^2 + \int_{\partial B_R} v_r^2 \dH{N-1}  }{ \int_{\partial B_1} v_r^2 \dH{N-1} } 
\leq C_4(R) \frac{ \int_{\partial B_R} v_r^2 \dH{N-1}  }{\int_{\partial B_1} v_r^2 \dH{N-1}  }.
\end{align}
	\textbf{Step 2.} \\
Hence the claim of this lemma will be proved once we have shown that for each $R>1$ 
\begin{align}\label{eq:doubling_bound}
	\frac{ \int_{\partial B_R} v_r^2 \dH{N-1}  }{\int_{\partial B_1} v_r^2 \dH{N-1}  } \leq C_5(R) \quad \text{ for every } r>0.
\end{align}
We assume towards a contradiction that this is not true, i.e. there is $R_0>1$ and a sequence $(r_k)_{k \in \N}$, $r_k \to \infty$ as $k \to \infty$ such that 
	\begin{align} \label{proof:unif_W_1_2_bound_assump}
		\frac{\int_{\partial B_{R_0}} v_{r_k}^2  \dH{N-1}}{\int_{\partial B_1} v_{r_k}^2 \dH{N-1}} \to \infty \quad \text{ as } k \to \infty.
	\end{align}

Let us now set for each $k \in \N$
\begin{align}
	\tilde{w}_{r_k}(x) := \frac{v_{r_k}(x)}{\sqrt{  \int_{\partial B_{R_0}} v_{r_k}^2 \dH{N-1}  } }.
\end{align}	
	From \eqref{eq:frequency_estimate_on_large_spheres} and a Poincaré-Lemma we infer that $(\tilde{w}_{r_k})_{k\in \N}$ is bounded in $W^{1,2}(B_{R_0})$. Passing to a subsequence if necessary,
		\begin{align} \label{def:w}
		\tilde{w}_{r_k} \rightharpoonup \tilde{w} \quad \text{ weakly in } W^{1,2}(B_{R_0})\text{ as }k\to\infty.
	\end{align}
	Furthermore we infer from the compact embeddings $W^{1,2}(B_{R_0})\hookrightarrow L^2(\partial B_{R_0})$ and $W^{1,2}(B_1)$ $\hookrightarrow L^2(\partial B_1)$ (recall that ${R_0} >1$)  that
	\begin{align} \label{eq:convergence_of_w_r_k_on_the_boundary}
	\begin{split}
		\tilde{w}_{r_k} \to \tilde{w} \quad &\text{ strongly in } L^2(\partial B_{R_0})\text{ and} \\
		\tilde{w}_{r_k} \to \tilde{w}   \quad &\text{ strongly in } L^2(\partial B_1)\text{ as }k\to\infty.
		\end{split}
	\end{align}
	By construction, $\int_{\partial B_{R_0}} \tilde{w}^2_{r_k}  \dH{N-1}=1$
for all $k \in \N$,
        so
	\begin{align} \label{eq:trace_of_w_on_B_1}
		\int \limits_{\partial B_{R_0}} \tilde{w}^2\dH{N-1} = 1  .
	\end{align}
	On the other hand, using \eqref{proof:unif_W_1_2_bound_assump} and \eqref{eq:convergence_of_w_r_k_on_the_boundary} we obtain that
	\begin{align}
		\int \limits_{\partial B_1} \tilde{w}^2  \dH{N-1} &\leftarrow 	\int \limits_{\partial B_1} \tilde{w}^2_{r_k}  \dH{N-1}= \frac{\int_{\partial B_1} v_{r_k}^2  \dH{N-1}}{\int_{\partial B_{R_0}} v_{r_k}^2  \dH{N-1}}  \to 0 \quad \text{ as } k \to \infty.
	\end{align} 
	Note that in dimension $N\ge 3$, $\tilde{w}$ is harmonic in $B_{R_0}$ since it has been constructed as the weak $W^{1,2}$-limit of $(\tilde{w}_ {r_k})_ {k\in \N}$ in \eqref{def:w}, and $\Delta \tilde{w} =0$ in $B_{R_0}$ up to the set $\{t e^N : t \geq 0  \}$ which is due to Definition \ref{def:solution} \eqref{PDE_asymptotics} of $2$-capacity zero . In dimension $N=2$, $\tilde{w}$ is harmonic in $B_1 \setminus \set {t e^N : t \geq 0}$ and zero on $B_1 \cap \set {t e^N : t \geq 0}$. 
	
	Consequently the maximum principle for harmonic functions implies that
	\begin{align}
		\tilde{w}\equiv 0 \quad \text{ in } B_ 1.
	\end{align}
	Since $w$, being harmonic in $\R^N \setminus \{ t e^N : t \geq 0   \}$, is analytic in that open set, it
        follows that $\tilde{w}\equiv 0$ in 
        $B_{R_0}$ contradicting
 \eqref{eq:trace_of_w_on_B_1}. 
	Therefore the assumption \eqref{proof:unif_W_1_2_bound_assump} must be false and the Lemma is proved.
\end{proof}

Lemma \ref{lem:w_r_in_W_1_2_loc} allows us to conclude that each limit of $w_r$ must be harmonic in $\R^N$:

\begin{prop} \label{prop:w_is_harmonic_for_N_geq_3}
	Let $N \geq 3$ and let $(w_r)_{r>0}$ be as defined in \eqref{def:w_r}. Then there is a sequence $r_k\to\infty$ such that
\begin{align} \label{eq:weak_convergence_of_w_r}
w_{r_k} \rightharpoonup w \quad \text{ weakly in } W^{1,2}_{\operatorname{loc}}(\R^N) \text{ as } k \to \infty,
\end{align}
and $w$ is harmonic in $\R^N$.
\end{prop}

\begin{proof}

From Lemma \ref{lem:w_r_in_W_1_2_loc} it follows that
\begin{align} 
	w_{r_k} \rightharpoonup w \quad \text{ weakly in } W^{1,2}_{\operatorname{loc}}(\R^N) \text{ as } k \to \infty.
\end{align}
Using the assumption on the blow-down in Definition  \ref{def:solution} \eqref{PDE_asymptotics} we obtain that
\begin{align}
	\Delta w = 0 \quad \text{ in } \R^N \setminus \set {t e^N \in \R^N : t \in \R }.
\end{align}
Since $\set {t e^N \in \R^N : t \in \R }$ is a set of $2$-capacity zero in dimensions $N \geq 3$, we infer that in these dimensions
\begin{align}
	\Delta w = 0 \quad \text{ in } \R^N.
\end{align}
\end{proof}

\begin{lem}[The limit $w$ is a quadratic polynomial] \label{lem:quadratic_growth_of_w}
	\mbox{} \\
	Let $N\ge 3$ and let $w$ be as defined in Proposition \ref{prop:w_is_harmonic_for_N_geq_3} \eqref{eq:weak_convergence_of_w_r}. Then $w$ is a harmonic polynomial of degree $\leq 2$.
\end{lem}

\begin{proof} 
	The strategy of the proof of this lemma will be to use the first frequency estimate \eqref{eq:first_frequency_estimate} in order to obtain a doubling that allows us to deduce that $w$ has at most quadratic growth at infinity. Then a Liouville argument implies that $w$ is a polynomial of degree $\leq 2$.  
	\\
First of all note that since $w\Delta w=0$ in $\R^N$,
\begin{align} \label{eq:harmonic_functions_dirichtlet_functional_represention}
	\int \limits_{B_1} \abs {\nabla w}^2 = \int \limits_{\partial B_1} w~ \nabla w \cdot x  \dH{N-1}  - \int \limits_{B_1} \underbrace{w \Delta w }_{=0}  =  \int \limits_{\partial B_1} w ~ \nabla w \cdot x  \dH{N-1} .
\end{align}
Let us now define for each $R>0$
\begin{align}
	y(R):= \int \limits_{\partial B_1} z^2_R  \dH{N-1} \quad \text{ where} \quad z_R(x) := w(Rx) \text{ for all } x \in \R^N.
\end{align}
Then the derivative of $y(R)$ satisfies
\begin{align} \label{eq:differential_equation_for_y}
	\frac{\dx{}}{\dx{R}} y(R) &= \int \limits_{\partial B_1} 2 z_R ~ \partial_R z_R   \dH{N-1}
	= 2 \int \limits_{\partial B_1} w(Rx) ~ \nabla w(Rx) \cdot x  \dH{N-1} \\
	&= \frac{2}{R} \int \limits_{\partial B_1} z_R ~ \nabla z_R \cdot x  \dH{N-1}	= \frac{2}{R} \int \limits_{B_1} \abs {\nabla z_R}^2,
\end{align}
where we have used  \eqref{eq:harmonic_functions_dirichtlet_functional_represention} in the last step. In order to deduce a differential inequality, we use that $z_R$, too, satisfies the first frequency estimate, i.e.
\begin{align}
	 \frac{\int_{B_1} \abs {\nabla z_R}^2}{\int_{\partial B_1} z_R^2  \dH{N-1}} \leq 2 \quad \text{ for all } R>0.
\end{align}
This may be verified as follows: From Lemma \ref{lem:first_frequency_estimate} and \eqref{eq:frequency_is_independent_of_harmonic_quadratic_polynomials} we deduce that for each $R$ and each $r>0$
\begin{align}
	2 &\geq \frac{\int_{B_1} \abs {\nabla (\tilde{v}_{rR} -h_r)}^2}{\int_{\partial B_1} (\tilde{v}_{rR} -h_r)^2  \dH{N-1}} 
	= R^2 ~ \frac{\int_{B_1} \abs {\nabla w_r(Rx)   }^2 }{ \int_{\partial B_1}  w_r^2(Rx)  \dH{N-1}    }.
\end{align}
Now weak convergence of $w_r$ as $r \to \infty$ (recall \eqref{eq:weak_convergence_of_w_r}) and lower semicontinuity of the Dirichlet-functional as well as the compactness of the trace embedding imply that for each $R>0$,
\begin{align}
2 \geq\ R^2 ~ \frac{\int_{B_1} \abs{\nabla w(Rx) }^2}{\int_{\partial B_1} w^2(Rx)  \dH{N-1}} = \frac{\int_{B_1} \abs {\nabla z_R}^2 }{\int_{\partial B_1} z_R^2  \dH{N-1}}.
\end{align}
Combining this frequency estimate with \eqref{eq:differential_equation_for_y} we obtain the following differential inequality for $y$:
\begin{align}
	\frac{\dx{}}{\dx{R}}~ y(R) = \frac{2}{R} \int \limits_{B_1} \abs{\nabla z_R}^2 \leq \frac{4}{R} \int \limits_{\partial B_1} z_R^2  \dH{N-1}= \frac{4}{R} ~y(R).
\end{align}
Consequently, for all $R\geq 1$
\begin{align} \label{eq:growth_of_y}
	y(R) \leq y(1) ~ R^4.
\end{align}
A similar estimate holds for the volume integral:
For each $R \geq 1$,
\begin{align}
	\int \limits_{B_1} z_R^2 &= \int \limits_\frac{1}{R}^1 \int \limits_{\partial B_\rho} z_R^2(x) \dH{N-1}(x) \dx{\rho} + \int \limits_{B_\frac{1}{R}} z_R^2  = \int \limits_\frac{1}{R}^1 \int \limits_{\partial B_1} \rho^{N-1}z_R^2( \rho x)  \dH{N-1}(x) \dx{\rho} +R^{-N} \int \limits_{B_1} w^2   \\&=  \int \limits_\frac{1}{R}^1 \int \limits_{\partial B_1} \rho^{N-1}z_{R\rho}^2( x)  \dH{N-1}(x) \dx{\rho} + R^{-N} \int \limits_{B_1} w^2 \leq \int \limits_\frac{1}{R}^1 \rho^{N-1} (\rho R)^4 y(1) \dx{\rho} +R^{-N}\int \limits_{B_1} w^2  \\
	&= \frac{y(1)}{N+4} \bra {1 - \frac{1}{R^{N+4}}} R^4 + R^{-N}\int \limits_{B_1} w^2 
	\leq C_{1} R^4, \label{eq:z_R^2_bound_in_average}
\end{align}
where we have used that $\rho R \geq 1$ for $\rho \in \pra {\frac{1}{R}, 1}$, \eqref{eq:growth_of_y} and $w \in W^{1,2}_{\operatorname{loc}}(\R^N)$.
Note that up to this point the proof holds for all $N\ge 2$.

We are going to combine (\ref{eq:z_R^2_bound_in_average}) with the mean value property of harmonic functions in order to obtain  a uniform estimate on the second derivatives.

Let $x_0 \in B_\frac{1}{8}$. Then for all $i,j \in \set {1, \dots, N}$,
by the mean value property for harmonic functions
\begin{equation} \label{eq:d_ij_z_R_first}
\abs { \partial_{ij } z_R(x_0) } \leq C \sup \limits_{\partial B_\frac{1}{8}(x_0)} \abs { \partial_i z_R} .
\end{equation}
Similarly we compute for $x \in B_\frac{1}{4}(0)$ (note that for all $x_0 \in B_\frac{1}{8}(0)$, $B_\frac{1}{8}(x_0) \subset B_\frac{1}{4}(0)$),
\begin{align}\label{eq:d_ij_z_R_second}
	\abs {\partial_i z_R(x)} &
        \leq 
C_1	\sup \limits_{\partial B_\frac{1}{8}(x)} \abs { z_R} 
	\leq   C_2	\sup \limits_{B_\frac{3}{8}(0)} \abs {z_R} \\
	& = 
C_3	\sup \limits_{y \in B_\frac{3}{8}(0)} \left|  \fint \limits_{B_\frac{1}{8}(y)} z_R    \right|
	 \leq   	C_4 \int \limits_{B_1} \abs {z_R}
	 \leq  
	 C_5 \sqrt{ \int \limits_{B_1} z_R^2}.
\end{align}
Combining \eqref{eq:d_ij_z_R_first} and \eqref{eq:d_ij_z_R_second} and using \eqref{eq:z_R^2_bound_in_average} we obtain that
\begin{align}
	\norm { \partial_{ij} z_R  }_{L^\infty (B_\frac{1}{8})} \leq 
C_6	  \sqrt{ \int \limits_{B_1} z_R^2} \leq C_7 R^2,
\end{align}
for all $R\geq 1$. Recalling that $\partial_{ij} z_R(x) =R^2 \partial_{ij} w(Rx)$, it follows that
\begin{align}
	R^2 ~ \norm { \partial_{ij} w  }_{L^\infty(B_\frac{R}{8})} \leq C_7 R^2.
\end{align}
Thus we arrive at the desired uniform bound
\begin{align}
\norm {D^2 w}_{L^\infty (B_\frac{R}{8})} \leq C_{ 2 } \text{ for all } R \geq 1.
\end{align}
Consequently Liouville's theorem implies that $D^2 w$ (being harmonic) is constant. This tells us that $w$ is a harmonic polynomial of degree $\leq 2$.
\end{proof}
\begin{rem}
  We conjecture that for $N=2$, methods developed in the present paper can be used to show that
  $$w(r,\theta) = -\Vert r^{3/2} \sin(3\theta/2)\Vert_{L^2(\partial B_1)}^{-1} r^{3/2} \sin(3\theta/2).$$
  \end{rem}
In order to obtain a nontrivial estimate on the asymptotics of $u$ we need to exclude quadratic growth of $w$ which is done in the following lemma.
\begin{lem}[The limit $w$ is an affine linear function] \label{lem:w_is_linear}
Let $N \geq 3$ and let $w$ be as defined in \eqref{eq:weak_convergence_of_w_r}. Then $w$ is a nonzero polynomial of degree $\le 1$.
\end{lem}

\begin{proof}
	From Lemma \ref{lem:quadratic_growth_of_w} we already know that $w$ is a harmonic polynomial of degree $\leq 2$. Therefore we may write
	\begin{align}
		w = h + \ell + c,
	\end{align}
	where $h$ is a harmonic, homogeneous polynomial of degree $2$, 
	$\ell=  b \cdot x$ ($b \in \R^N$) is a linear function and $c \in \R$ is a constant. Note that the fact that $\int_{B_1} w^2 \dH{N-1}=1$ implies that $w$ is not the zero polynomial.
	We prove the claim of the lemma in two steps:\\
\textbf{Step 1.} \emph{$h$ is independent of $x_N$.}\\						
By results of 	\cite[Proof of Case 3 of Theorem II]{CaffarelliKarpShahgolian_Annals_2000} 
$\partial_N u $ does not change sign, and   by our assumption that  $p = p(x') $  and  $h_r$ are
 independent of $x_N$, $\partial_N w_r$ does not change sign either. Hence the limit
 $\partial_N w $ does not change sign.
But then $h$ cannot contain terms of the form $x_N^2$ or $x_j x_N$, $j \in \set {1, \dots, N-1}$. Hence $h$ is independent of $x_N$ as claimed.\\
\textbf{Step 2.} \emph{$h \equiv 0$.}\\
		By definition  \eqref{def:v_r} and \eqref{def:h_r}, $\Pi' w_r = 0 $ for all $r>0$, and in the limit $\Pi' w = 0 $, implying together with Step 1 that $h\equiv 0$ and that $w$ is a polynomial of degree $1$.\end{proof}

        \begin{rem}
          We will prove later that $\ell$ in the previous proof is not zero.
          \end{rem}


\section{First estimate on the growth of the coincidence set $\cC$} \label{section:first_estimate_on_the_coincidence_set}

We are now in a position to obtain a first estimate on the growth of $\cC$ as $x_N \to \infty$. 
\begin{prop}[First estimate on $\cC$] \label{lem:first_estimate_on_coincidence_set}
	\mbox{} \\
Let $N \geq 3$ and let $u$ be a solution in the sense of Definition \ref{def:solution}. Then for each $\delta \in (0,1)$ there is a number $a=a(\delta)\in (0,+\infty)$ such that
	\begin{enumerate}[i)]
			\item \label{item:growth_of_coincidence_set}
		\begin{align}
			\cC \cap \set {x_N >a} \subset \set {\abs{x'}^2 <x_N^{1+\delta}} \text{ and }
		\end{align}
			\item \label{item:boundedness of coincidence_set_around_zero}
	\begin{align} 
		\cC \cap \set {x_N \leq a} \text{ is bounded.}
		\end{align} 
		
	\end{enumerate}
\end{prop}
\begin{proof}
	\mbox{} 
		i) The main tool in proving this claim is a quantitative version of the doubling we have already used in \eqref{eq:doubling_bound}. In order to derive a nontrivial bound on the coincidence set $\cC$, our previous analysis on the asymptotic behavior of solutions ---i.e. that $w$ is affine linear--- is essential.
		Let us define the scaled function
		\begin{align} \label{def:v^r}
		\begin{split}
		v^r(x) &:= u^r(x) -p^r(x') -{h}^r_r(x') \\&= u(rx) -p(rx') - h_r(rx') =r^2 v_r(x)
		 \end{split}
		\end{align}
		and for all $r>0$
		\begin{align}
		f(r) := \sqrt {  \int \limits_{\partial B_1} \bra {v^r}^2  \dH{N-1}  }.
		\end{align}
Let us now note that the fact that $h_{2r} := \Pi' \tilde{v}_{2r}$ (cf.  \eqref{def:h_r}) implies that for all $r>0$
\begin{align}
	\int \limits_{\partial B_1} (\tilde{v}_{2r} -h_{2r})^2 \dH{N-1} \leq \int \limits_{\partial B_1} (\tilde{v}_{2r} - h_r)^2 \dH{N-1}.
\end{align}
	This observation together with	Lemma \ref{lem:w_is_linear} implies the following doubling
		\begin{align}
		\frac{\int_{\partial B_1} \bra {v^{2r}}^2  \dH{N-1}}{\int_{\partial B_1} \bra {v^{r}}^2  \dH{N-1}} &= \frac{\int_{\partial B_1} \pra { (2r)^2 ( \tilde{v}_{2r} - h_{2r} )  }^2 \dH{N-1}  }{ \int_{\partial B_1} (v^r)^2 \dH{N-1}   } \leq \frac{\int_{\partial B_1} \pra { (2r)^2 ( \tilde{v}_{2r} - h_{r} )  }^2 \dH{N-1}  }{ \int_{\partial B_1} (v^r)^2 \dH{N-1}   } 
		\\&= 2^{1-N} \frac{\int_{\partial B_2} (r^2 v_r)^2 \dH{N-1}}{\int_{\partial B_1} (r^2 v_r)^2 \dH{N-1}}  =	
		2^{1-N} \int \limits_{\partial B_2} \bra {w_r}^2  \dH{N-1} \\
		 &\to 
		 2^{1-N} \int \limits_{\partial B_2} w^2 \dH{N-1} 
	\leq 4 \int \limits_{\partial B_1} w^2  \dH{N-1} = 4 \\
		\end{align}
		as $r \to  \infty$.
		It follows that for all $\delta \in (0,1)$ there is $r_0(\delta)<+\infty$ such that for all $r> r_0(\delta)$, $f(2r) \leq 2^{1+\delta} f(r)$.
		Iterating this estimate we obtain  for all $k \in \N$ and $r > r_0(\delta)$
		\begin{align}
		f \bra {2^k r} \leq 2^{(1+\delta)k} f(r).
		\end{align}
	We deduce for all $k \in \N$ and all $r \in \pra {2^k r_0, 2^{k+1}r_0}$ that
		\begin{align}
		f(r) &\leq 2^{(1+\delta)k} \sup \limits_{\rho \in \pra {r_0, 2 r_0}} f(\rho) 
	  \leq C_1(r_0) \> r^{1+\delta} \sup \limits_{\rho \in \pra {r_0, 2 r_0}} f(\rho)=: C_2(\delta)\> r^{1+\delta}. \label{eq:L^2_growth}
		\end{align}
				This allows us to estimate the asymptotic thickness of the coincidence set $\cC$ as $x_N \to \infty$.  To this purpose we need to improve the estimate on the above  squared average  to a pointwise estimate. We do this using 
a sup-mean-value-inequality for subharmonic functions.
Before going into details let us remind the reader that due to the definition of $h_r$ in \eqref{def:h_r},
\begin{align}
	\norm {h_r}_{L^2(\partial B_1)} = \norm {  \Pi' (u_r-p)}_{L^2(\partial B_1)} \leq \norm {u_r-p}_{L^2(\partial B_1)} \to 0 \text{ as } r\to\infty.
\end{align}
Invoking that $\cP_2'$ is a finite dimensional vector space where all norms are equivalent, all coefficients of $h_r$ must vanish as $r \to \infty$. Since $p$ is non-degenerate in $x'$ we obtain that for all sufficiently large $r$,
				\begin{align} \label{eq:absorb_h_r_into_p}
		\abs {h_r(x')} \leq \frac{1}{2} p(x') \quad \text{ for all } x' \in \R^{N-1}.
		\end{align}
		
		Remembering from \eqref{def:v^r} that 
		\begin{align}
		-v^r(x)=p^r({x'})+h_r^r({x'})-u^r(x) = p(r{x'})+h_r(r{x'}) -u(rx),
		\end{align}
		we know that
		\begin{align}
		\max \set {p^r+{h}^r_r-u^r,0}
		\end{align}
		is a non-negative, subharmonic function, so that
by a sup-mean-value-property of {non-negative} subharmonic functions,
		\begin{align} \label{eq:sup_mean_for_estimate_growth_of_coincidence_set}
		 \sup \limits_{B_\frac{1}{2}}  \max \set { p^r +h_r^r -u^r,0 }\leq C(N) \sqrt{ \int \limits_{\partial B_1} \max \set { p^r +h_r^r -u^r,0 }^2  \dH{N-1} }. 
		\end{align}
		Let now $x \in \cC$ be such that $r:= 4 \abs {x}$ is sufficiently large.
		Combining \eqref{eq:L^2_growth}, \eqref{eq:absorb_h_r_into_p} and \eqref{eq:sup_mean_for_estimate_growth_of_coincidence_set} we obtain that
		\begin{align}
		C(N) C_{2}(\delta ) r^{1+\delta} &\geq \sup \limits_{B_\frac{r}{2}} \max \set { p+h_r-u,0  } \geq \max \set { p(x') +h_r(x') -u(x) ,0   } \\
		&=  \max \set {p(x') +h_r(x') ,0} \geq \max \set { \frac{1}{2} p(x') ,0   } = \frac{1}{2} p(x').
		\end{align} 
		This means that for  all sufficiently large $r$ and every $x \in \cC \cap \set {\abs {x} = \frac{r}{4}} $,
		\begin{align}
		\frac{c_p}{2} \abs {x'}^2 &\leq  \frac{1}{2} p(x')  \leq 4^{1+\delta} C(N) \> C(\delta) \> \abs {x}^{1+\delta} \leq  8^{1+\delta} C(N) \> C(\delta)  \bra {  \abs {x'}^{1+\delta} + \abs{x_N}^{1+\delta} },
		\end{align}
		where $c_p$ is  defined in \eqref{eq:def_of_c_p}.
		It follows that there is a constant $C<+\infty$ such that for sufficiently large $|x'|$,
		\begin{align}
		\abs{x'}^2 \leq C x_N^{1+\delta},
		\end{align}
		and we obtain \ref{item:growth_of_coincidence_set})
                choosing a slightly larger $\delta$ and choosing the number $a$ sufficiently large.\\
ii) follows from Proposition \ref{prop:known_results} \eqref{prop:known_results_item_4}.
\end{proof}
\section{{The Newtonian potential expansion of $u$}} \label{section:Newton_potential_expansion_of_u}

In this section we are going to use the growth estimate for the coincidence set in Proposition \ref{lem:first_estimate_on_coincidence_set} in order to show that the Newton-potential of the coincidence set $\cC$ is well-defined and has subquadratic growth. This allows us to do a Newton-potential expansion of the solution $u$. It is this Newton-potential expansion which will allow us to control the asymptotics of the solution up to a constant outside a small region around the coincidence set $\cC$.

\begin{lem}[Newtonian potential of $\cC$] \label{lem:existence_of_newton_potential}
	\mbox{} \\
Let $u$ be a solution in the sense of Definition \ref{def:solution} and let $N \geq 6$. Then
	\begin{enumerate}[i)]
		\item 
		The Newtonian potential $V_\cC$ of $\cC$ is well-defined and locally bounded.
		 \\
		\item \label{item:subquadratic_growth_of_Newton_potential_of_coincidence_set}
		$V_\cC(x)$ grows subquadratically as $\abs {x} \to \infty$, i.e. 
		\begin{align}
			\frac{V_\cC(x)}{|x|^2} \to 0 \text{ as } |x|\to\infty.
		\end{align}
	\end{enumerate}
\end{lem}

\begin{proof}
	\mbox{}  
        To prove i), it suffices 	to check  that the Newtonian potential of $\cC \setminus B_R$ is well-defined and locally bounded for some $R>0$.
        Let $M<+\infty$, $\delta := 1/10$ and let $R$ be sufficiently large such that $\cC \setminus B_R \subset \set {y \in \R^N :  y_N >\max(a(\delta), 2M)} $, where $a(\delta)$ is the constant defined in Proposition \ref{lem:first_estimate_on_coincidence_set}. Then for 
$$
T_\delta := \left\{\abs{y'}^2 < y_N^{1+\delta} \right\}  \cap \big\{y_N > \max \set { a(\delta) , 2M} \big\}$$
and every $x$ such that $|x_N|\le M$,
	\begin{align}
		\int \limits_{\cC \setminus B_R} \frac{1}{ \abs {x-y}^{N-2} } \dx{y} &\leq 	
		\int \limits_{T_\delta } \frac{1}{ \abs {x-y}^{N-2} } \dx{y}
		\leq 	\int \limits_{T_\delta } \frac{1}{ \abs {x_N-y_N}^{N-2} } \dx{y} 
		\leq \int \limits_{a(\delta)}^\infty \frac{1}{\abs {\frac{y_N}{2}}^{N-2}} \abs {B'_{y_N^{(1+\delta)/2}}} \dx{y_N} \\&=  2^{N-2} \abs {B'_1} \int \limits_{a(\delta)}^{+\infty} y_N^{-N+2 + \frac{1+\delta}{2}(N-1)} \dx{y_N}.
			\end{align}
			The last integrand is integrable 
                        for $\delta := 1/10$ and $N \geq 6$. 
It follows that the Newtonian potential of $\cC \setminus B_R$ is well-defined and locally bounded.
			
			Next we prove statement ii).  
			 Let $\delta$ be as defined above, let $a(\delta)$ be the constant defined in Proposition \ref{lem:first_estimate_on_coincidence_set} and let $R<+\infty$ be such that $\cC \cap \set {y_N < a(\delta)} \subset B_R$.
Define 
\begin{align}
 &P_1:= \set {\abs {y'}^2 < y_N^{1 + \delta } }\cap \set {y_N < x_N -x_N^\frac{23}{24} 
 },
 P_2:= B_{2 x_N^{23/24}}(0,x_N)
 \\
\label{decomp}
 &\text{and }P_3:= \set {\abs {y'}^2 < y_N^{1 + \delta }}\cap \set { y_N > x_N +x_N^\frac{23}{24}  } .
\end{align}
		Then for $x_N$ large enough,
			\begin{align} \label{eq:overestimation_coincidence_set_for_subquadratic_growth}
				\cC \subset B_R \cup P_1 \cup P_2 \cup   P_3,
			\end{align}
which in turn implies that
			\begin{align}
			V_\cC(x)  &\leq V_{B_R}(x) + V_{P_1} (x)+ V_{P_2}(x) + V_{P_3}(x) \leq V_{B_R}(x) + V_{P_1} ((0,x_N)) + V_{P_2}((0,x_N))  + V_{P_3}((0,x_N))  .
			\end{align}
			For fixed  $R$,
			$	V_{B_R}(x) \to 0 $  as $ \abs {x} \to \infty$. 
Furthermore
\begin{align}
\frac{1}{\alpha_N}V_{P_1}((0,x_N)) &\leq  \int \limits_{P_1} \frac{1}{ \abs {x_N-y_N}^{N-2} } \dx{y}
\leq 
\bra {x_N^\frac{23}{24} }^{2-N} \abs {B_1'} \int \limits_0^{x_N} y_N^{\frac{1+\delta}{2}(N-1)} \dx{y_N} \\
&=  \abs {B_1'} \frac{1}{\frac{11}{20} (N-1)+1}~  x_N^{\frac{23}{24}(2-N) + \frac{11}{20} (N-1)+1} \quad \to \ 0,
\end{align}
as $x_N \to \infty$ due to the assumption $N \geq 6$. Next,
			
\begin{align}
				V_{P_2} (0,x_N) &= V_{B_{2 x_N^{23/24}}(0,x_N)}((0,x_N)) = \alpha_N \int \limits_{ B_{2 x_N^{{23/24}}  }} \frac{1}{ \abs {y}^{N-2}  } \dx{y} 
				\\ &= 
				\alpha_N \int \limits_0^{2 x_N^{23/24}}  \rho^{2-N} \abs {\partial B_1} \rho^{N-1} \dx{\rho} = 2  \alpha_N  \abs {\partial B_1}  x_N^\frac{23}{12}
			\end{align}
			which has subquadratic growth. Finally,
			\begin{align}
				V_{P_3}(0,x_N) &\leq \alpha_N \int \limits_{\set {\abs { y' }^2 < y_N^{1+\delta} \land y_N > x_N +  x_N^{23/24}}} \frac{1}{\abs {x_N-y_N}^{N-2}}   \dx{y}\\
				&=  \int \limits_{x_N^{23/24}}^{x_N}y_N^{2-N} |B_1'| \bra { y_N +x_N  }^{(N-1)  \frac{1+\delta}{2}  } \dx{y_N}    + \int \limits_{x_N}^{+\infty} y_N^{2-N} |B_1'| (y_N+x_N)^{(N-1) \frac{1+\delta}{2} }   \dx{y_N} \\
				&\leq   \int \limits_{x_N^{23/24}}^{x_N}y_N^{2-N} |B_1'|\bra { 2 x_N  }^{(N-1)  \frac{1+\delta}{2}  } \dx{y_N}    + \int \limits_{x_N}^{+\infty} y_N^{2-N} |B_1'|(2y_N)^{(N-1) \frac{1+\delta}{2} }   \dx{y_N}   \\
				&= |B_1'|(2x_N)^{(N-1) \frac{1+\delta}{2}  } ~\frac{x_N^{(3-N) \frac{23}{24}   }- x_N^{3-N}   }{N-3}     + |B_1'|\frac{2^{\frac{1+\delta}{2}(N-1) +1 }}{N-5-\delta(N-1)}  x_N^\frac{5-N+\delta(N-1)}{2}   \\
				&\leq C \> x_N^{-\frac{1}{8}}   ,
		\end{align}
		 for some constant $C<+\infty$ and $x_N$ large enough, due to the assumption $N \geq 6$. 
	 This tells us that the growth of the Newtonian potential is dominated by the part $P_2$. Thus we have established the subquadratic growth of the Newtonian potential of $\cC$ as $\abs {x} \to \infty$.
\end{proof}
	
\begin{prop}[Newtonian potential expansion]
Let $u$ be a solution in the sense of Definition \ref{def:solution} and let $N \geq 6$.  Then the expansion
	\begin{align}
		u = p +\ell +c +V_\cC
	\end{align}
	holds, where $p$ is the quadratic polynomial in Definition \ref{def:solution} \eqref{PDE_asymptotics}, $\ell$ is a linear function such that $\partial_N \ell <0$, and $c$ is a constant.
\end{prop}

\begin{proof}
It is well known that $V_\cC$ is a strong solution in $W^{2,p}_{\operatorname{loc}}(\R^N)$ of
\begin{align}
	\Delta V_\cC = - \chi_\cC \text{ in } \R^N.
\end{align}
Let us furthermore set 
\begin{align}
	v := u-p \quad \text{ in } \R^N.
\end{align}
Then $v$ solves the same equation as $V_\cC$, i.e. $v\in W^{2,p}_{\operatorname{loc}}(\R^N)$ is a strong solution of
\begin{align}
\Delta v = - \chi_\cC \text{ in } \R^N.
\end{align}
Hence $v-V_\cC$ is harmonic in $\R^N$, and from Definition \ref{def:solution} \eqref{PDE_asymptotics} and Lemma \ref{lem:existence_of_newton_potential}
we know that $v-V_\cC$ has subquadratic growth. This allows us to apply Liouville's theorem
to obtain that
\begin{align}
	v-V_\cC = \ell +c,
\end{align}
where $\ell$ is a linear function and $c$ is a constant. Thus we have proved 
\begin{align} \label{eq:Newton_potential_expansion_of_u}
	u = p + \ell + c + V_\cC\quad \text{ in } \R^N.
\end{align}
What remains to be shown is that
\begin{align} \label{eq:ell_has_sign}
	 \partial_N \ell <0. 
\end{align}
Since $0 \in \cC$ (cf. Definition \ref{def:solution}) let $x^1:=-e^N$.
It follows that $\abs{y} < \abs {x^1-y}$  and
\begin{align}
	 \frac{1}{\abs {x^1-y}^{N-2}} < \frac{1}{\abs {y}^{N-2}}\quad \text{ for all } y \in \cC.
\end{align}
Consequently 
 $V_\cC(x^1) < V_\cC({0})$.
Employing the Newtonian potential expansion \eqref{eq:Newton_potential_expansion_of_u}, we obtain that 
\begin{align}
	0<u(x^1) &= \ell(x^1) +c +V_\cC(x^1) < -\ell(e^N) +c +V_\cC(0) = u(0) - \ell(e^N) =-\ell(e^N).
\end{align}
\end{proof}
\section{Existence of suitable paraboloid solutions} \label{section:existence_of_paraboloid_solutions}

While it is not difficult to show that each paraboloid gives rise to {\em some} solution of the obstacle problem (e.g. using a sequence of ellipsoids converging to the given paraboloid) it is a different matter altogether to prove that, given $p$ and $\ell$, there exists a solution of the obstacle problem with a paraboloid as coincidence set that has precisely $p+\ell$ as asymptotic limit at infinity.
The following result showing this existence is related to the homeomorphism (mapping the ellipsoids onto the class of quadratic polynomials describing the asymptotic behavior of the solution at infinity) constructed in \cite[Proof of (5.4)]{DiBenedettoFriedman} in the case of compact coincidence set.

\begin{thm}[Existence of paraboloid solutions with prescribed linear part] \label{prop:existence_of_paraboloid_solutions} 
	\mbox{} \\
	Let $N   \geq 6$.
	For each $(b_1,\dots,b_{N+1}) \in (0,\infty)^{N} \times \R$ such that $\sum_{j=1}^{N-1} b_j = \tfrac{1}{2}$ there is $(a_1,\dots, a_N) \in (0,+\infty)^ {N-1}  \times  \R$ such that
	\begin{align}
		V_{P_{\bf a}}(x)= - \sum \limits_{j = 1}^{N-1} {b_j}{x_j^2} + b_N x_N + b_{N+1} \quad \text{ in } P_{\bf a},
	\end{align}
	where 
	\begin{align}
		P_{\bf a} &:= \set { (x',x_N) \in \R^N : x_N\ge -a_N, x' \in \sqrt{x_N + a_N} E'_{\bf a'}   } \\
		  E'_{\bf a'} & := \set {  x' \in \R^{N-1} : \sum \limits_{j=1}^{N-1} \frac{x_j^2}{a_j^2} \leq 1  }.
	\end{align}
	Furthermore 
	\begin{align}
		u_{P_{\bf a}}(x) := p_{\bf b}(x') - b_N x_N -b_{N+1} +V_{P_{\bf a}}(x)
	\end{align}
	solves
	\begin{align}
 u_{P_{\bf a}} \geq 0 \text{ in } \R^N \quad ,\quad \Delta u_{P_{\bf a}} = \chi_{ \set {u_{P_{\bf a}} >0} } \text{ in } \R^N \quad , \quad \set {u_{P_{\bf a}} =0} = P_{\bf a} \\\quad \text{ and } \quad \frac{u_{P_{\bf a}}(rx)}{r^2} \to p_{\bf b}(x') ~\text{ uniformly on } \partial B_1 \text{ as } r \to \infty,
 	\end{align}
 	where $p_{\bf b}(x') := \sum_{j=1}^{N-1} b_j x_j^2 $ and $E'_{\bf a'}$ is the (up to scaling) \emph{unique} ellipsoid corresponding to the polynomial $p_{\bf b}(x')$ in the sense that there is $\lambda >0$ such that
 	\begin{align}
 		V'_{ \lambda E'_{\bf a'}}(x') = 1- p_{\bf b}(x') \quad \text{ for all  } x' \in \lambda E'_{\bf a'}. 
 \end{align}
\end{thm}

The proof is based on the following Lemma which is a consequence  of the analysis of the Newtonian potential of ellipsoids carried out in \cite{DiBenedettoFriedman}.

\begin{lem}[Existence of suitable ellipsoids]\label{lem:existence_of_suitable_ellipsoids}
		\mbox{} \\
Let $N \geq 3$. Then for each non-degenerate, symmetric, homogeneous quadratic polynomial $q(x) := \sum_{j=1}^N q_j x_j^2$ with $q_j >0$ for all $j \in \set {1, \dots, N}$ and $ \sum_{j=1}^N q_j =\tfrac{1}{2}$ and each constant $c>0$, there exists a unique ellipsoid, centered at the origin, $$E= \set {  x \in \R^{N} : \sum_{j=1}^{N} \frac{x_j^2}{a_j^2} \leq 1  } $$  such that
\begin{align}
	V_E(x) = c- q(x) \quad \text{ for all } x \in E.
\end{align}
\end{lem}

\begin{proof}[Proof of Lemma \ref{lem:existence_of_suitable_ellipsoids}]
The proof is a corollary to a result by 
DiBenedetto and Friedman in \cite{DiBenedettoFriedman} (see the proof of (5.4) therein). They show that for each polynomial $q$ as above there is an ellipsoid $\tilde{E}$, centred at the origin, and some constant $\tilde{c} >0$ such that
	\begin{align}
		V_{\tilde{E}} (x) = \tilde{c} - q(x) \quad \text{ for all } x \in \tilde
		E.
	\end{align}

A direct computation shows that the Newtonian potential obeys the scaling law 
\begin{align} \label{eq:scaling_of_Newton_potential}
	 V_{\beta \tilde{E}}(x) = \beta^2 V_{\tilde{E}} \bra {\frac{x}{\beta}}
\text{ for all } \beta >0.
\end{align}
Thus for all $x \in \beta \tilde{E}$,
\begin{align}
	V_{\beta \tilde{E}}(x) = \beta^2 \tilde{c} -q(x).
\end{align}
	Choosing $\beta := \sqrt{\frac{c}{\tilde{c}}}$ and $E:= \beta \tilde{E}$ finishes the proof.
	\\
It remains to prove uniqueness of the ellipsoid $E$. The comparison and Hopf-principle argument in \cite{ellipsoid} (see step 2 and 3 in the proof of Theorem 2 therein) implies that the ellipsoid $E$ is unique up to scaling, and prescribing the constant $c = V_E(0)$ rules out this degree of freedom. 
\end{proof}

\begin{proof}[Proof of Theorem \ref{prop:existence_of_paraboloid_solutions}]
	\mbox{} \\
	\textbf{Step 1.} \emph{Construction of a suitable sequence of ellipsoids} \\
Let us define for each $n \in \N$
\begin{align}\label{eq:qn}
	q^n(x) := \bra { 1- \frac{2}{n^2}  }  p_{\bf b}(x') + \frac{1}{n^2} x_N^2 \quad   \quad \text{ and } \quad    c_n := \bra { \frac{b_N n}{2}  }^2 >0   .
\end{align}
Then Lemma \ref{lem:existence_of_suitable_ellipsoids} implies  that there is a centered ellipsoid $\tilde{E}^n$ such that
\begin{align} \label{eq:Newton_potential_on_E_n}
	V_{\tilde{E}^n} = c_n -q^n \quad \text{ on } \tilde{E}^n.
\end{align}
In order to produce the prescribed linear term in the Newtonian potential expansion we translate $\tilde{E}^n$ by $\tau_n e^N$, where $\tau_n :=  \frac{b_N}{2} n^2$, i.e.
\begin{align}
E^n := \tilde{E}^n + \tau_n e^N.
\end{align}
We infer from \eqref{eq:Newton_potential_on_E_n} that for all $x \in E^n$   
\begin{align}
V_{E^n}(x) &= V_{\tilde{E}^n} \bra {x - \tau_n e^N  } = c_n - \bra { 1- \frac{2}{n^2}  }   p_{\bf b}(x') - \frac{1}{n^2} x_N^2 + \frac{2 \tau_n}{n^2} x_N - \frac{1}{n^2} \tau_n^2 \\
 &= b_N x_N -q^n(x).\label{eq:rephrased-pot}
\end{align}
	\textbf{Step 2.} \emph{Switching to the obstacle problem and passing to the limit.} \\
In order to be able to use known results and techniques from the analysis of the obstacle problem we make use of the close relation between null quadrature domains and the obstacle problem (cf. \cite{KarpMargulis_free_boundaries}). Defining for $n \in \N$
\begin{align} \label{eq:Obstacle_problem_paraboloid_solution_construction}
	u_n := q^n -b_N x_N + V_{E^n} \quad \text{ in } \R^N,
\end{align}
$u_n$ is a non-negative solution of the obstacle problem
\begin{align}
	\Delta u_n = \chi_{\set {u_n >0}} ~ \text{ in } \R^N \quad \text{ and } \set {u_n = 0} = E^n
\end{align}
(see for example \cite[Theorem II]{CaffarelliKarpShahgolian_Annals_2000}).
Using the non-negativity of the Newtonian potential together with
\eqref{eq:rephrased-pot} and \eqref{eq:qn}
we obtain that for all $x \in E^n$, $n \geq 2$
\begin{align}
p_{\bf b}(x') \leq 2 ~b_N x_N.
\end{align}
Since this estimate is independent of $n$, there is a paraboloid $\tilde{P}=\{ p_{\bf b}(x') \leq 2~ b_N x_N\}$ such that 
\begin{align} \label{eq:En_is_contained_in_paraboloid}
	E^n \subset  \tilde{P} \text{ for every } n \in \N \setminus \{1\}.
\end{align}
From Lemma \ref{lem:existence_of_newton_potential} we know that the Newtonian potential of $\tilde{P}$ is well-defined and locally bounded in dimension $N \geq 6$. As
\begin{align}
	0\le V_{E^n} \leq V_{\tilde{P}} \quad \text{ in } \R^N \text{ for all }n \in \N \setminus \{1\},
\end{align}
we obtain that $(V_{E^n})_{n\in \N}$ and $(u_n)_{n\in \N}$ are bounded in $L^\infty_{\operatorname{loc}}(\R^N)$.
From $L^p$-theory we infer that
for each $p \in [1,\infty)$ and each $\alpha\in (0,1)$,
\begin{align}
	(u_n)_{n\in \N} \quad \text{ is bounded in } W^{2,p}_{\operatorname{loc}}(\R^N)\cap C^{1,\alpha}_{\operatorname{loc}}(\R^N).
\end{align}
Thus there is a subsequence (again labeled $(u_n)_{n \in \N}$) such that
\begin{align} \label{eq:C1alpha_limit_of_ellipsoidal_solutions}
	u_n \to u \quad \text{ in } C_{\operatorname{loc}}^{1,\alpha}(\R^N) \text{ as }n\to\infty,
\end{align}
and (cf. \cite[Proposition 3.17]{PetrosyanShahgholianUraltseva_book}) $u$ is a non-negative solution of the obstacle problem, i.e. $u$ solves
\begin{align}
	\Delta u = \chi_{\set {u>0}} \quad \text{ in } \R^N.
\end{align}
\\
	\textbf{Step 3.} \emph{Identification of the coincidence set of $u$ and switching back to the Newtonian potential expansion.} \\
In order to identify the coincidence set of $u$ we will pass to the limit in the Newton-potential expansion \eqref{eq:Obstacle_problem_paraboloid_solution_construction} of $u_n$. To this end recall that each ellipsoid $E^n$ is the sublevel set of a polynomial, so $E^n$ is of the form
\begin{align} \label{eq:ellipsoidal_expression_as_sublevelset}
E^n = \set {  \sum \limits_{j=1}^{N-1} \frac{x_j^2}{B_{j,n}^2} + \frac{(x_N - \tau_n)^2}{B_{N,n}^2} \leq 1      },
\end{align} 
where $B_{j,n} \in (0, \infty)$ are the semiaxes of $E^n$ and $\tau_n$ is the translation in $e^N$-direction as defined in step 1 (for all $n \in \N$ and $j \in \set {1, \dots, N}$).

Since for all $n \in \N$, $E^n$ is defined by finitely many coefficients $(B_{1,n}, \dots, B_{N,n}, \tau_n)$ which converge (passing if necessary to a subsequence) in $[0,\infty]^{N+1}$ we infer that
\begin{align} \label{eq:convergence_of_characteristic_functions_of_E_n}
	\chi_{E^n} \to \chi_M \quad \text{pointwise almost everywhere in } \R^N \text{ as } n \to \infty,
\end{align}
where $M \subset \tilde{P}$ is some measurable set. Using
\begin{align}
	\chi_{E^n}(y) \abs {x-y}^{2-N} \leq \chi_{\tilde{P}}(y) \abs {x-y}^{2-N} \quad \text{ for all } x,y \in \R^N
\end{align}
we obtain by dominated convergence that
\begin{align}
	V_{E^n} \to V_{M} \quad \text{pointwise in } \R^N \text{ as } n \to \infty.
\end{align}
Combining this fact with \eqref{eq:Obstacle_problem_paraboloid_solution_construction} and \eqref{eq:C1alpha_limit_of_ellipsoidal_solutions} we obtain the Newton-potential expansion 
\begin{align} \label{eq:Newton_potential_expansion_of_limit_of_ellipsoidal_solutions}
	u(x) = p_{\bf b}(x') -b_N x_N +V_M(x) \quad \text{ for all } x \in \R^N.
\end{align}
It remains to identify the set $M$. First of all, from \eqref{eq:Newton_potential_expansion_of_limit_of_ellipsoidal_solutions} we infer that $M$ has non-vanishing Lebesgue-measure, i.e. $\abs{M} >0$: Otherwise $V_M \equiv 0$ in $\R^N$ which combined with \eqref{eq:Newton_potential_expansion_of_limit_of_ellipsoidal_solutions} would contradict the fact that $u$ is non-negative in $\R^N$.

Note that $E^n \subset \tilde{P}$ implies that for all $n \in \N$, $0 \leq B_{N,n} \leq \tau_n$. Combining this observation with \eqref{eq:convergence_of_characteristic_functions_of_E_n} and the fact that $M$ has positive measure we obtain that the ellipsoids $E^n$ cannot vanish towards infinity in the $e^N$-direction and therefore, recalling that by the definition 
$\tau_n = \frac{b_N}{2} n^2$, passing if necessary to
a subsequence,
\begin{align} \label{eq:convergence_of_translation_by_N-semiaxis}
1 \leq	\frac{\tau_n}{B_{N,n}} \to 1 \quad \text{ as } n \to \infty.
\end{align}
Let us now rewrite \eqref{eq:ellipsoidal_expression_as_sublevelset} as 
\begin{align}
	E^n = \set {  \sum \limits_{j=1}^{N-1} \frac{\tau_n}{B_{j,n}^2} x_j^2+ \frac{\tau_n}{B_{N,n}^2} x_N^2 -2 \bra {\frac{\tau_n}{B_{N,n}}}^2 x_N \leq \pra {1-  \bra {\frac{\tau_n}{B_{N,n}}}^2 }   \tau_n   }.
\end{align}
We claim that $  \pra {1-  \bra {\frac{\tau_n}{B_{N,n}}}^2 }   \tau_n \leq 0$ is bounded in $n$. Assume towards a contradiction that $  \pra {1-  \bra {\frac{\tau_n}{B_{N,n}}}^2 }   \tau_n$ is unbounded, i.e. that there is a subsequence such that $  \pra {1-  \bra {\frac{\tau_n}{B_{N,n}}}^2 }   \tau_n \to -\infty$. Then by \eqref{eq:convergence_of_translation_by_N-semiaxis}
\begin{align}
	E^n \subset \set {  -2 \bra { \frac{\tau_n}{B_{N,n}} }^2 x_N \leq \pra {1-  \bra {\frac{\tau_n}{B_{N,n}}}^2 }   \tau_n     } \to \emptyset \quad \text{ as } n \to \infty,
\end{align}
which is incompatible with \eqref{eq:convergence_of_characteristic_functions_of_E_n} and the fact that $\abs{M}>0$. Hence, passing if necessary to a subsequence, $\pra {1-  \bra {\frac{\tau_n}{B_{N,n}}}^2 }   \tau_n \to c \in (-\infty,0]$ as $n \to \infty$.

Passing if necessary to another subsequence, $\frac{\tau_n}{B_{j,n}^2} \to B_j \in [0, \infty]$ as $n \to \infty$ for all $j \in \set {1, \dots , N-1}$. We claim  that $B_j \in (0,\infty)$ for all $j \in \set {1, \dots, N-1}$. 
Assume first towards a contradiction that there is $i \in \set {1, \dots, N-1}$ and a subsequence such that $\frac{\tau_n}{B_{i,n}^2} \to +\infty$ then
\begin{align}
E^n \subset  \set {  \frac{\tau_n}{B_{i,n}^2} x_i^2 - 2 \bra { \frac{\tau_n}{B_{N,n}} }^2 x_N \leq    \pra {1-  \bra {\frac{\tau_n}{B_{N,n}}}^2 }   \tau_n    }\to E^0 \subset \set {x_i =0},
\end{align}
which poses a contradiction to \eqref{eq:convergence_of_characteristic_functions_of_E_n} and the fact that $\abs{M}>0$.
To finish the proof assume towards a contradiction that there is $i \in \set {1, \dots, N-1}$ such that $\frac{\tau_n}{B_{i,n}^2} \to 0$. Then for all $n \in \N$
\begin{align}
	E^n &\supset \set {  \frac{\tau_n}{B_{i,n}^2} x_i^2 \leq  \pra {1-  \bra {\frac{\tau_n}{B_{N,n}}}^2 }   \tau_n ~  , ~ x_N=0, x_j = 0 , j \neq i  } \\
	&\to \set {x_N=0, x_j= 0 , j \neq i} \quad \text{ as } n \to \infty.
\end{align}
But this is impossible since from \eqref{eq:En_is_contained_in_paraboloid} we know that $E^n$ must be contained in the paraboloid $\tilde{P}$ for all $n \in \N$.

Summing up we conclude that (passing if necessary to a subsequence)
\begin{align}
	\chi_{E^n} \to \chi_M ~ \text{pointwise a.e.  as } n \to \infty,  \text{ where } M=\set { \sum \limits_{j =1}^{N-1} B_j x_j^2 -2 x_N \leq c   },
\end{align}
$ B_j \in (0,\infty)$ for all $j \in \set {1, \dots , N-1}$ and $c \in (-\infty,0]$.

Translating the paraboloid in the $e^N$-direction such that the constant part in the expansion agrees with $b_{N+1}$ finishes this step. \\
	\textbf{Step 4.} \emph{Identification of the sectional ellipsoids of the limit coincidence set.}\\
	We now know that $\{u=0\} = P_{\bf a}$, where 
	\begin{align}
P_{\bf a} := \{ (x',x_N) \in \R^N : x_N\ge -a_N, x' \in \sqrt{x_N + a_N} E'_{\bf a'}   \} .
		\end{align}
It remains to show that the sectional ellipsoid $E_{\bf a'}' \subset \R^{N-1}$  is up to scaling the unique ellipsoid $\tilde{E}' \subset \R^{N-1}$ from Lemma \ref{lem:existence_of_suitable_ellipsoids} 
such that 
\begin{align}
	V'_{\tilde{E}'}(x') = V'_{\tilde{E}'}(0) - p_{\bf b'}(x')  \quad \text{ for all } x' \in \tilde{E}'.
\end{align}
In order to prove this, let us define the following blow-down with moving center of the paraboloid solution $u$ (cf. \eqref{eq:Newton_potential_expansion_of_limit_of_ellipsoidal_solutions}) we have constructed:
\begin{align}
	\text{ for all  } k \in \N, x \in \R^N: \quad u_k(x)  := \frac{u(x^k + r_k x)}{r_k^2},
\end{align}
where $x^k := (0, x_N^k)$, $x_N^k \to \infty$ as $k \to  \infty$ and $r_k := \sqrt{x_N^k}$. Then by Calderon-Zygmond theory we infer that (up to taking a subsequence)
\begin{align}
	u_k \to \tilde{u} \quad \text{ in } C_{\loc}^{1,\alpha}(\R^N) \text{ as } k \to \infty
\end{align}
and it is known that (cf. \cite[Proposition 3.17]{PetrosyanShahgholianUraltseva_book}) $\tilde{u}$ is a global solution of the obstacle problem. Furthermore $\{ u=0 \} = P_{\bf a}$ implies that $\{ \tilde{u} =0\} = E_{\bf a'}' \times \R$. Therefore from Proposition \ref{prop:known_results} \eqref{prop:known_results_item_3} we infer that $\tilde{u}$ is independent of $x_N$, i.e. 
\begin{align}
	\tilde{u}(x) = \tilde{u}'(x') := \tilde{u}(x',0) \quad \text{ for all } x \in \R^N.
\end{align}
It is known (see  proof of Theorem II, Case 2, and 3 in \cite{CaffarelliKarpShahgolian_Annals_2000}) 
that every blow-down limit of any global solution of the obstacle problem is either a  half-space solution or a homogeneous polynomial of degree $2$ satisfying $\Delta q \equiv 1$. 
It is further known (see  proof of Theorem II, Case 2 in \cite{CaffarelliKarpShahgolian_Annals_2000}) that if  the blow-down of any solution is a half-space solution then 
the solution itself has to be a half-space solution. Since $\tilde{u}$ is not a half-space solution we can apply
Lemma \ref{lem:preservation_of_blow_down} and infer that the blow-down of $u$ coincides with that of $\tilde{u}$ and hence
\begin{align}
	\lim \limits_{\rho \to \infty} \frac{\tilde{u}(\rho x)}{\rho^2}  = p_{\bf b}(x') = \lim \limits_{\rho \to \infty} \frac{u(\rho x)}{\rho^2} \quad \text{ in } L^\infty(\partial B_1).
\end{align}
Therefore the function $\tilde{v}'(x') := \tilde{u}'(x') - p_{\bf b}(x') - V_{{E}_{\bf a'}'}'(x')$ is harmonic and has subquadratic growth and using Liouville's theorem we conclude that there is a linear function $\tilde{\ell}: \R^{N-1} \to \R$ and a constant $\tilde{c} \in \R$ such that $\tilde{v}' = \tilde{\ell} + \tilde{c}$ and hence
\begin{align}
	\tilde{u}'(x') = p_{\bf b}(x') + \tilde{\ell}(x') + \tilde{c} + V'_{{E}_{\bf a'}'}(x').
\end{align}
The fact that $\{  \tilde{u}' =0\} = E_{\bf a'}'$ implies that $\tilde{c} = V'_{{E}_{\bf a'}'}(0)$ and $\tilde{E}_{a'}'$ being centered implies that $\nabla \tilde{\ell} = \nabla V_{E_{\bf a'}'}'(0)=0$. Putting everything together we have that
\begin{align}
	V'_{E_{\bf a'}'}(x') = V'_{E_{\bf a'}'}(0) - p_{\bf b}(x') \quad \text{ for all  } x' \in E'_{\bf a'}.
\end{align}
Using the scaling property of the Newton potential (cf. \eqref{eq:scaling_of_Newton_potential}) there is $\lambda >0$ such that 
\begin{align}
	V'_{\lambda {E}_{\bf a'}'}(0) = V'_{\tilde{E}'}(0) \quad \text{ and} \quad V'_{\lambda {E}_{\bf a'}'}(x') = V'_{\lambda {E}_{\bf a'}'}(0) - p_{\bf b}(x') \quad \text{ in } \lambda E'.
\end{align}
The uniqueness in Lemma \ref{lem:existence_of_suitable_ellipsoids} implies that $\lambda E_{\bf a'}' = \tilde{E}'$.
\end{proof}

\section{Decay of the Newtonian potential of $P$ outside a narrow neighborhood of $P$}

The asymptotic behavior of the Newtonian potential at infinity will be crucial in our proof of Main Theorem**. 
In this section we are going to show decay of the Newtonian potential of $P$ towards infinity outside a narrow neighborhood of $P$. 
\begin{lem}
\label{lem:Newton_potential_vanishes}
	Let $N \geq 6$, $\gamma >0$,
	\begin{align}
		P := \set { (y',y_N) \in \R^N : \abs {y'} < \gamma y_N^\frac{1}{2}   },
	\end{align}
	and define for each $\mu > \frac{25}{72}$
	\begin{align}
		P^\mu := \set { (y',y_N) \in \R^N : \abs {y'} < \gamma y_N^{\frac{1}{2}+\mu}   }.
	\end{align}
	Then
	\begin{align} \label{eq:behaviour_Newton_potential_paraboloid_i}
		\sup \limits_{x \in (\R^N \setminus P^\mu) \cap \set {x_N >k} } V_P(x) \to 0 \quad \text{ as } k \to \infty
	\end{align}
	and  
\begin{align} \label{eq:behaviour_Newton_potential_paraboloid_ii}
		\sup \limits_{x \in (\R^N \setminus B_k ) \cap \set {x_N \leq \frac{k}{2}} } V_P(x) \to 0 \quad \text{ as } k \to \infty.
	\end{align}
\end{lem}
The lemma states that the Newtonian potential of $P$ vanishes outside a narrow neighborhood of $P$. (Note that $k$ in \eqref{eq:behaviour_Newton_potential_paraboloid_i} and \eqref{eq:behaviour_Newton_potential_paraboloid_ii} are independent.)

\begin{proof}
 As in \eqref{decomp}
 we decompose $P$ up into a set of points that are close to $x$ and
 the complement of that set, and
 we estimate the  Newtonian potential of each set individually.
	
	As $P$ is axially symmetric and $V_P(\lambda x'+x_N e^N)$ is a decreasing function of $\abs{\lambda}$ we obtain that
	\begin{align}
		\sup \limits_{x \in (\R^N \setminus P^\mu) \cap \set {x_N = k} } V_P(x) = V_P ( \gamma k^{\frac{1}{2} + \mu} e^1 + k e^N) .
	\end{align} 
	Furthermore,
	$	P  = P_1 \cup P_2 \cup P_3$ where 
	\begin{align}
	P_1 &= \set { \abs {y'} < \gamma y_N^\frac{1}{2} \land y_N < x_N -x_N^\frac{8}{9}  },   \\
	P_2&=  \set { \abs {y'} < \gamma y_N^\frac{1}{2} \land \abs{x_N -y_N} \leq x_N^\frac{8}{9}  }\text{ and} \\
	P_3&=   \set { \abs {y'} < \gamma y_N^\frac{1}{2} \land  y_N > x_N +x_N^\frac{8}{9}   } .
	\end{align}
Using this decomposition,
	$
		V_P=
		V_{P_1}
                +V_{P_2}
                +V_{P_3}.
	$
The first term satisfies
	\begin{align}
		V_{P_1} (\gamma k^{\frac{1}{2} +\mu} e^1 +k e^N )  &\leq \alpha_N \int \limits_0^{k - k^\frac{8}{9}}   (k-y_N)^{2-N} \gamma^{N-1} \abs{B_1'}  \bra {y_N^\frac{1}{2}}^{N-1} \dx{y_N}
		\leq C_1 \> k^{\frac{8}{9}(2-N) +\frac{1}{2}(N+1)} 
	\end{align}
which vanishes as $k \to \infty$ by the assumption $N \geq 6$.
Concerning the second term we obtain for large $k$ that
\begin{align}
	V_{P_2} (\gamma k^{\frac{1}{2} +\mu} e^1 +k e^N )  &\leq \alpha_N \int \limits_{P_2} \frac{1 }{ \abs{ \gamma k^{\frac{1}{2} +\mu} -y_1}^{N-2}  } \dx{y} 
	\leq \alpha_N \bra { \frac{\gamma}{2} k^{\frac{1}{2} +\mu}  }^{2-N} \int \limits_{k -k^\frac{8}{9} }^{k + k^\frac{8}{9}} \abs{B_1'} \bra {\gamma \sqrt{y_N}}^{N-1} \dx{y_N} \\
	&\leq \alpha_N \bra { \frac{\gamma}{2} k^{ \frac{1}{2} +\mu  }   }^{2-N} \abs {B_1'} \gamma^{N-1} (2k)^\frac{N-1}{2} 2 k^\frac{8}{9} 
	\leq C_2 \> k^{ \bra {\frac{1}{2} +\mu}(2-N) + \frac{1}{2}(N-1) + \frac{8}{9}   },
\end{align}
where the right-hand side vanishes as $k \to \infty$ for each $N\geq 6$ and $\mu > \frac{25}{72}$. With regard to the last term we get
\begin{align}
  &V_{P_3}( \gamma k^{\frac{1}{2} +\mu} e^1 +k e^N ) \leq \alpha_N \int \limits_{k+k^\frac{8}{9}}^{2k} \abs{k-y_N}^{2-N} \gamma^{N-1} y_N^\frac{N-1}{2} \abs{B_1'} \dx{y_N} \\& \quad \quad\quad+ \alpha_N \int \limits_{2k}^{+\infty} \abs{k-y_N}^{2-N} \gamma^{N-1} y_N^\frac{N-1}{2} \abs{B_1'} \dx{y_N} \\
	&\leq \alpha_Nk^{\frac{8}{9} (2-N)} \int \limits_{k+k^\frac{8}{9}}^{2k} \gamma^{N-1} y_N^\frac{N-1}{2} \abs{B_1'} \dx{y_N}  + \alpha_N\int \limits_k^{+\infty} y_N^{2-N} \gamma^{N-1} \bra {y_N+k}^\frac{N-1}{2} \abs{B_1'} \dx{y_N} \\
	&\leq \alpha_N\gamma^{N-1} \abs{B_1'} \frac{2}{N+1} 2^\frac{N+1}{2} k^{\frac{8}{9}(2-N) + \frac{1}{2}(N+1)} + \frac{\alpha_N\gamma^{N-1} \abs{B_1'} 2^\frac{N+1}{2} }{N-5} k^{\frac{5-N}{2}},
\end{align}
where the right-hand side vanishes as $k \to \infty$ for each $N\geq 6$ and $\mu > \frac{25}{72}$. This finishes the proof of \eqref{eq:behaviour_Newton_potential_paraboloid_i}.

Finally, we prove \eqref{eq:behaviour_Newton_potential_paraboloid_ii}. For $k$ large enough and every $x \in (\R^N \setminus B_k) \cap \set {y_N \leq \frac{k}{2}}$,
	\begin{align}
	V_P(x) &=  \alpha_N \int \limits_{P \cap \set {y_N \leq k} } \frac{1}{ \abs {x-y}^{N-2}  } \dx{y} + \alpha_N \int \limits_{P \cap \set {y_N \geq k} } \frac{1}{ \abs {x-y}^{N-2}  } \dx{y} \\
	&\leq  \alpha_N \int \limits_0^k  \bra {\frac{k}{2}}^{2-N}   \abs { B_1' } \bra { \gamma y_N^\frac{1}{2}  }^{N-1} \dx{y_N}  +  \alpha_N \int \limits_k^\infty  \bra {y_N - \frac{k}{2}}^{2-N} \abs {B_1'} \bra { \gamma y_N^\frac{1}{2}  }^{N-1} \dx{y_N} \\
	&\leq  \alpha_N \int \limits_0^k  \bra {\frac{k}{2}}^{2-N}   \abs {B_1' } \bra { \gamma y_N^\frac{1}{2}  }^{N-1} \dx{y_N}  + \alpha_N \int \limits_k^\infty  \bra {\frac{y_N}{2}}^{2-N} \abs {B_1'} \bra { \gamma y_N^\frac{1}{2}  }^{N-1} \dx{y_N} \\
	&\leq \alpha_N \frac{\abs { B_1'} (2 \gamma)^{N-1} }{N+1} k^{\frac{5}{2} - \frac{N}{2} } + \alpha_N \frac{ (2 \gamma)^{N-1} \abs {B_1'}}{N-5} k^{\frac{5}{2} - \frac{N}{2}} \quad 
	\to 0 \quad \text { as } k \to \infty.
	\end{align}
\end{proof}


\section{Proof of Main Theorem** via a comparison principle with insufficient information on the boundary
} \label{section:comparison}

In this  section we shall finish the proof of Main Theorem**.
Unlike the compact case in which the unknown coincidence set can be touched by an ellipsoid from the outside (cf.  \cite{ellipsoid}), it does not seem to be feasible to prove ---using only the knowledge we have gathered so far--- that the unknown coincidence set contains/is contained in a paraboloid. Instead we will prove that the unknown solution is on a large part of $\partial B_R$ greater than a known paraboloid solution and that the difference of the two solutions satisfies a one-sided estimate on the complement of that large part. The combination of those two estimates will lead to a comparison principle.\\
{\em Proof of Main Theorem** }\\
\noindent
\textbf{Step 1.} \emph{Construction of a comparison solution.} \\ 
Let us recall (cf. \eqref{eq:Newton_potential_expansion_of_u}) that
\begin{align}
u = p(x') + \ell(x) + V_\cC(x) + c \text{ in } \R^N.
\end{align}
Employing Theorem  \ref{prop:existence_of_paraboloid_solutions}
and translating if necessary 
we find a paraboloid $P$ such that $P\cap \{ x_N\le 0\}=\{ 0\}$ and that
\begin{align}
  u_P := p(x') + \ell(x) + V_P(x) +c_P\text{ in } \R^N
\end{align}
 is a solution of the obstacle problem; here $c_P$ is a constant.

Let us define for $\lambda \geq 0$ the translated paraboloid
\begin{align}
 P_\lambda := P - \lambda e^N
\end{align}
and
\begin{align}
	u_{P_\lambda}(x) := u_P(x+\lambda e^N).
\end{align}
Then
\begin{align}
	u_{P_\lambda}(x) &= p(x') + \ell(x) +V_{P_\lambda}(x) + \lambda \ell(e^N) + c_P \text{ in } \R^N, 
\end{align}
and since $V_\cC(x) \geq 0$,
\begin{align} \label{eq:expansion_of_difference_between_u_lambda_and_u}
	u_{P_\lambda}(x) -u(x) 
	\leq V_{P_\lambda}(x)  + \lambda \ell(e^N) +c_P-c\text{ in } \R^N.
\end{align}

\noindent
\textbf{Step 2.} \emph{Comparison for every $\lambda>\bar{\lambda} := (c_P-c)/(-\ell(e^N))$.} \\
Our aim is to compare $u_{P_\lambda}$ and $u$ for sufficiently large $\lambda$. To this end we will apply a sup-mean-value-inequality for non-negative subharmonic functions to
\begin{align}
	z^r(x) := z(rx), \text{ where } z:= \max \set { u_{P_\lambda} -u ,0      } \geq 0.
\end{align}
As, due to the fact that $u$ and $u_{P_\lambda}$ solve a semilinear PDE of the form $\Delta u=g(u)$ with $g$ non-decreasing, $z^r$ is a subharmonic function, 
so that
\begin{align} \label{sup_boundary_value_inequality}
  \sup \limits_{B_\frac{1}{2}} z^r \leq C(N) \fint \limits_{\partial B_1} z^r \dH{N-1} 
  \text{ for all } r \in (0,+\infty).
\end{align}
Let  $\gamma< +\infty$ be such that
\begin{align}
	P \subset  \tP := \set {  (y',y_N) \in \R^N : \abs {y'} \leq \gamma\sqrt{y_N}}  .
\end{align}
It follows
that
\begin{align} \label{eq:paraboloids_are_contained}
	\tP_\lambda := \tP - \lambda e^N \supset P_\lambda .
\end{align}
Choosing $\mu := \frac{7}{20} > \frac{25}{72}$ and $\tP^\mu$ as in Lemma \ref{lem:Newton_potential_vanishes} we set
\begin{align}
	\tP^\mu_\lambda := \tP^\mu - \lambda e^N.
\end{align}
By \eqref{eq:ell_has_sign}, $\ell(e^N) < 0$. This allows us to choose
 $\lambda_0>0$ sufficiently large such that
\begin{align}\label{minustwo}
c_P-c+\lambda_0 \ell(e^N) <0.
\end{align}
In the remainder of this step, we will prove $u_{P_{\lambda}} \leq u$
for each $\lambda$ such that $c_P-c+\lambda \ell(e^N) <0$, in particular
for $\lambda=\lambda_0$.
First, \eqref{eq:paraboloids_are_contained} and Lemma \ref{lem:Newton_potential_vanishes} tells us that there is $r_0<+\infty$ such that for all $r>r_0$
 \begin{align}
 	V_{P_{\lambda}} \leq V_{\tP_{\lambda}} < -(c_P-c+\lambda \ell(e^N))\quad \text{ on } \partial B_r \setminus \tP^\mu_{\lambda}.
 \end{align}
 So for $r>r_0$,
 \begin{align}
 	\max \set { u_{P_{\lambda}} -u ,0   } =0 \quad \text{ on } \partial B_r \setminus \tP^\mu_{\lambda}.
 \end{align}
Combining this with \eqref{eq:expansion_of_difference_between_u_lambda_and_u} we  estimate the right--hand side of \eqref{sup_boundary_value_inequality} as 
\begin{align}
	\fint \limits_{\partial B_1} z^r  \dH{N-1}   &
	 \leq  \frac{1}{ \abs {\partial B_r}  } \int \limits_{\partial B_r  \cap \tP^\mu_{\lambda}   } \max \set {V_{P_\lambda} + \lambda \ell(e^N) + c_P-c,0}  \dH{N-1}  
	\leq \frac{1}{ \abs {\partial B_r}  } \int \limits_{\partial B_r  \cap \tP^\mu_{\lambda}   }  V_{P_{\lambda}} \dH{N-1}.  
\end{align}
In the remainder of this step we will estimate the right-hand side.
By a direct calculation we obtain that for sufficiently large $r$,
\begin{align} \label{eq:estimate_of_sphere_and_paraboloid}
	\partial B_r \cap \tP^\mu_{\lambda} \subset \set {  r - 5 \gamma^2 r^{2 \mu} < y_N < r    }.
\end{align}
Let us decompose and estimate $V_{\tP_{\lambda}}$ as follows:
\begin{align}
	V_{\tP_{\lambda}} &= \alpha_N \int \limits_{\tP_{\lambda}} \frac{1}{ \abs {x-y}^{N-2}  } \dx{y} \\
	&\le  \alpha_N\Bigg( \int \limits_{ \tP_{\lambda,1}} \frac{1}{ \abs {x_N-y_N}^{N-2}  } \dx{y} +  \int \limits_{ \tP_{\lambda,2}} \frac{1}{ \abs {x'-y'}^{N-2}  } \dx{y} +  \int \limits_{ \tP_{\lambda,3}} \frac{1}{ \abs {x_N-y_N}^{N-2}  } \dx{y} \Bigg) ,
\end{align}
where
\begin{align}
	\tP_{\lambda,1} &:= \tP_{\lambda} \cap \set { y_N < r -6\gamma^2 r^{2 \mu  }}, \\
		\tP_{\lambda,2} &:= \tP_{\lambda} \cap \set { r -6\gamma^2 r^{2 \mu  }   < y_N <  r +6\gamma^2 r^{2 \mu  }  }, \\
			\tP_{\lambda,3} &:= \tP_{\lambda} \cap \set { y_N >  r +6\gamma^2 r^{2 \mu  }  }.
\end{align}
In order to avoid unnecessary confusion we will in the following always use $y$ as the variable of integration in the Newtonian potential integral and $x$ will always be on $\partial B_r \cap \tP_\lambda^\mu$
so that $x_N$ satisfies the bound in 
\eqref{eq:estimate_of_sphere_and_paraboloid}.

Using  the scaling and growth properties of Newtonian potential like integrals on bounded sets as well as Fubini's Theorem we obtain for the second part of the decomposition $\tP_{\lambda,2}$, that
\begin{align}
 \int \limits_{ {\tP_{\lambda,2}}} \frac{1}{ \abs {x'-y'}^{N-2}  } \dx{y} \leq \int \limits_{r-6\gamma^2 r^{2 \mu }}^{r+6\gamma^2 r^{2 \mu }} W_{2 \gamma y_N^\frac{1}{2} B_1'}(x') \dx{y_N},
\end{align}
 where for any bounded set $M \in \R^{N-1}$ we define for all $x' \in \R^{N-1}$
 \begin{align}
 	W_M(x') := \int \limits_M \frac{1}{\abs {x'-y'}^{N-2}} \dx{y'}.
 \end{align}
 A calculation shows that $W$ obeys for all $\beta>0$ and bounded and measurable $M \subset \R^{N-1}$ the following scaling law: 
 \begin{align} \label{eq:sacling_of_W}
 	W_{\beta M}(x') = \beta W_M \bra { \frac{x'}{\beta}  } \quad \text{ for all } x' \in \R^{N-1},
 \end{align}
By another direct calculation we obtain that
 \begin{align}
 	\abs{x'}^{N-2}W_M(x') \to \abs{M} \quad \text{ uniformly as } \abs{x'} \to \infty,
 \end{align}
 which implies that there is $C(M)<+\infty$ such that for all $x' \in \R^{N-1}$
 \begin{align} \label{eq:asymptotics_of_W}
 W_M (x') \leq C(M) \abs {x'}^{2-N}.
 \end{align}
 Combining \eqref{eq:sacling_of_W} and \eqref{eq:asymptotics_of_W} this allows us to estimate
 \begin{align}
 W_{2 \gamma y_N^\frac{1}{2} B_1'}(x') = 2 \gamma \sqrt{y_N} W_{B'} \bra {\frac{x'}{2 \gamma \sqrt{y_N}}} \leq 2 \gamma \sqrt{y_N} C(B_1') \abs {\frac{x'}{2 \gamma \sqrt{y_N}}}^{2-N}.
 \end{align}
 Consequently, for sufficiently large $r$,
 \begin{align}
  &\int \limits_{ {\tP_{\lambda,2}}} \frac{1}{ \abs {x'-y'}^{N-2}  } \dx{y} \leq (2 \gamma)^{N-1} C(B_1') \abs {x'}^{2-N}\int \limits_{r-6\gamma^2 r^{2 \mu }}^{r+6\gamma^2 r^{2 \mu }} y_N^{\frac{N-1}{2}} \dx{y_N}  \\ &\leq (2 \gamma)^{N-1} C(B_1') \abs{x'}^{2-N} (2r)^\frac{N-1}{2} (12 \gamma^2 r^{2\mu})
  = C_1(N,\gamma) \> \abs {x'}^{2-N} r^{\frac{N-1}{2} +2\mu}. \label{eq:Newton_potential_on_bad_set_main_part_inner_estimate}
 \end{align}
In order to estimate integrals over the sphere cap $\partial B_r \cap \tP^\mu_{\lambda}$, we are going to use for sufficiently large $r$ and every non-negative Borel-measurable function $f$, that
 \begin{align} 
 \int \limits_{\partial B_r \cap \tP^\mu_{\lambda} } f(x') \dH{N-1}(x) &= \int \limits_{B_{2 \gamma r^{\frac{1}{2} + \mu}}'} f(x') \frac{r}{ \sqrt{ r^2 - \abs {x'}^2   }  } \dx{x'} \leq 2  \int \limits_{B_{2 \gamma r^{\frac{1}{2} + \mu}}'} f(x')  \dx{x'}.\label{eq:integration_on_part_of_sphere}
 \end{align} 
 Hence, employing \eqref{eq:Newton_potential_on_bad_set_main_part_inner_estimate} for large $r$ we get that
 \begin{align}
  &\int \limits_{ \partial B_r \cap \tP^\mu_{\lambda}   }              \int \limits_{ {\tP_{\lambda,2}}} \frac{1}{ \abs {x'-y'}^{N-2}  } \dx{y}               \dH{N-1}(x)   
   \leq 2 C_1 \> r^{\frac{N-1}{2} +2\mu } \int \limits_{B_{2 \gamma r^{\frac{1}{2} + \mu}}'} \abs {x'}^{2-N} \dx{x'} \\
 	&=  2 C_1 \> r^{\frac{N-1}{2} +2\mu } \int \limits_0^{2 \gamma r^{\frac{1}{2} + \mu}} \abs {\partial B'_1} \rho^{N-2} \rho^{2-N} \dx{\rho} = 4 C_1 \> \abs {\partial B'_1} \gamma r^{\frac{N}{2} +3\mu}.
 \end{align}
  It follows that
 \begin{align}
 \frac{1}{  \abs {\partial B_r}   }		\int \limits_{ \partial B_r \cap \tP^\mu_{\lambda}   }              \int \limits_{ {\tP_{\lambda,2}}} \frac{1}{ \abs {x'-y'}^{N-2}  } \dx{y}               \dH{N-1}(x) \leq C_2(N,\gamma) \> r^{-\frac{N}{2} +1 +3 \mu },
 \end{align}
which vanishes as $r \to \infty$ by the assumption that $N \geq 6$ and $\mu = \frac{7}{20}$.
 
Concerning $\tP_{\lambda,1}$,
we estimate  for sufficiently large $r$ and for all $x \in \partial B_r \cap \tP^\mu_{\lambda}$ (using \eqref{eq:estimate_of_sphere_and_paraboloid})
 \begin{align}
 	 \int \limits_{ \tP_{\lambda,1}} \frac{1}{ \abs {x_N-y_N}^{N-2}  } \dx{y}  &= \int \limits_{-\lambda}^{r - 6\gamma^2 r^{2\mu}}  \frac{1}{ \abs {x_N-y_N}^{N-2}  } \abs {B_1'} \bra {  \gamma (y_N + \lambda)^\frac{1}{2}  }^{N-1} \dx{y_N} \\ 
 	 &\leq \bra { \gamma^2 r^{2 \mu}  }^{2-N} \abs {B_1'}  \gamma^{N-1}  \int \limits_0^{2r} y_N^\frac{N-1}{2} \dx{y_N} 
 	 \leq C_{3}(N,\gamma) \> r^{ \bra {2 \mu} (2-N)  + \frac{1}{2} (N+1)   } .
 \end{align}
As this estimate is uniform in $x \in \partial B_r \cap \tP^\mu_{\lambda}$ we obtain
 \begin{align}
 &\frac{1}{  \abs {\partial B_r}   }		\int \limits_{ \partial B_r \cap \tP^\mu_{\lambda}   }  \int \limits_{ \tP_{\lambda,1}} \frac{1}{ \abs {x_N-y_N}^{N-2}  } \dx{y}  \dH{N-1}(x)  \leq C_3 \> r^{ \bra {2 \mu} (2-N)  + \frac{1}{2} (N+1)   } \frac{\abs{\partial B_r \cap \tP^\mu_{\lambda}  }}{ \abs{\partial B_r} } .
 \end{align}
 From
 \eqref{eq:integration_on_part_of_sphere} we infer that
\begin{align} \label{eq:estimate_of_small_part_of_sphere_to_whole_sphere}
 \frac{\abs{\partial B_r \cap \tP^\mu_{\lambda}  }}{ \abs{\partial B_r} } \leq 2^N\frac{\abs{B_1'} \gamma^{N-1}}{\abs{\partial B_1}} r^{\bra{-\frac{1}{2} +\mu}(N-1)}
\end{align}
so that
\begin{align}
	\frac{1}{  \abs {\partial B_r}   }		\int \limits_{ \partial B_r \cap \tP^\mu_{\lambda}   }  \int \limits_{ \tP_{\lambda,1}} \frac{1}{ \abs {x_N-y_N}^{N-2}  } \dx{y}  \dH{N-1}(x)    \leq C_4 \> r^{\mu(3-N) +1  }
\end{align}
which vanishes as $r \to \infty$ by the assumption that $N\geq 6$ and $\mu = \frac{7}{20}$.

Concerning $\tP_{\lambda,3}$,
 we similarly estimate for sufficiently large $r$ and for every $x \in \partial B_r \cap \tP^\mu_{\lambda}$ (using \eqref{eq:estimate_of_sphere_and_paraboloid})
 \begin{align}
 	 &\int \limits_{ \tP_{\lambda,3}} \frac{1}{ \abs {x_N-y_N}^{N-2}  } \dx{y}  = \int \limits_{r + 6\gamma^2 r^{2\mu } }^{+\infty} \abs {x_N-y_N}^{2-N}   \abs {B_1'} \bra {  \gamma \bra {  y_N + \lambda   }^\frac{1}{2}  }^{N-1} \dx{y_N} \\
 	 &\leq   \int \limits_{r + 6\gamma^2 r^{2 \mu}}^{2r}  \abs{ x_N-y_N  }^{2-N}  \abs{B_1'} \bra { \gamma (y_N + \lambda)^\frac{1}{2}   }^{N-1} \dx{y_N}  
 	 +  \int \limits_{2r}^{{+\infty}}  \abs{ x_N-y_N  }^{2-N}  \abs{B_1'} \bra { \gamma (y_N + \lambda)^\frac{1}{2}   }^{N-1} \dx{y_N}       \\
 	 &\leq  \bra {6\gamma^2 r^{2\mu}}^{2-N} \abs{B_1'}  2^\frac{N-1}{2} \gamma^{N-1}  \int \limits_{r + 6\gamma^2 r^{2 \mu}}^{2r}  y_N^\frac{N-1}{2} \dx{y_N}  
 	 +  
	  \int \limits_r^{+\infty} y_N^{2-N} \abs {B_1'} \bra {\gamma (3 y_N)^\frac{1}{2}}^{N-1} \dx{y_N}      \\
 	 &\leq C_5(N) \> r^{(2\mu) (2-N) + \frac{1}{2}(N+1)}.
 \end{align}
 As this estimate is uniform in $x \in \partial B_r \cap \tP^\mu_{\lambda}$ we may use \eqref{eq:estimate_of_small_part_of_sphere_to_whole_sphere} as in the estimate of $\tP_{\lambda,2}$
  to arrive at
 \begin{align}
 	\frac{1}{\abs {\partial B_r}} \int \limits_{ \partial B_r \cap \tP^\mu_{\lambda}   }  	 \int \limits_{ \tP_{\lambda,3}} \frac{1}{ \abs {x_N-y_N}^{N-2}  } \dx{y} \dH{N-1}(x)  \leq  C_6(N) \> r^{\mu(3-N) +1  }
 \end{align}
 which vanishes as $r\to \infty$ by the assumption that $N\geq 6$ and $\mu = \frac{7}{20}$.
 
 So the sup-mean-value-inequality \eqref{sup_boundary_value_inequality} tells us that for each $\epsilon >0$ there is $r_0(\epsilon) >0$ such that for every $r > r_0(\epsilon)$
 \begin{align}
  \sup \limits_{B_\frac{1}{2} } z^r \leq \epsilon \quad \text{ and } \quad \sup \limits_{ B_\frac{r}{2} } z \leq \epsilon .
 \end{align}
  We conclude that 
 \begin{align}
 	z \equiv 0 \quad \text{ in } \R^N
 \end{align}
and consequently that
 \begin{align}
 	u_{P_{\lambda}} \leq u \quad \text{ in } \R^N
 \end{align}
 and
 \begin{align}
 	\cC \subset P_{\lambda}
 \end{align}
 (see Figure \ref{fig:coincidence_set_in_paraboloid}). \\
 	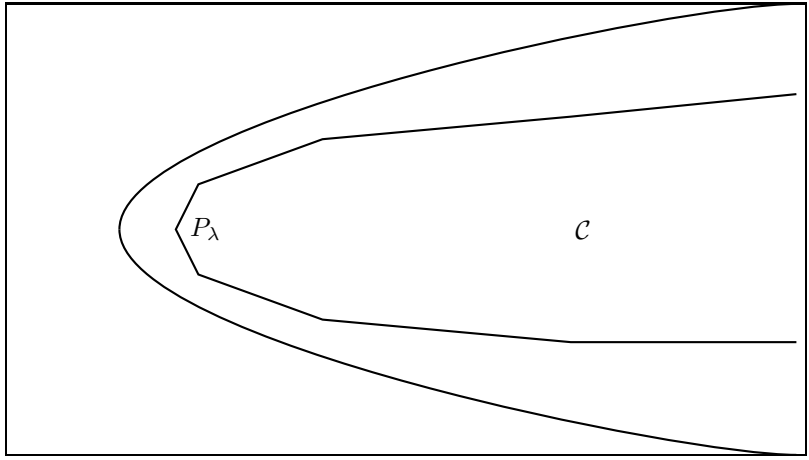
\begin{figure}[h!]
 	\centering	
 	\frame{
 		\psset{xunit=1.5cm,yunit=1.5cm}
 		\begin{pspicture*}
 		(-6,-2)(1,2)
 		\psbezier(1,-2)(0,-2)(-5,-1)(-5,0)
 		\psbezier(1,2)(0,2)(-5,1)(-5,0)
 		\psline(1,1.2)(-1,1)(-3.2,0.8)(-4.3,0.4)(-4.5,0)(-4.3,-0.4)(-3.2,-0.8)(-1,-1)(1,-1)
 		\rput[7,0.5](0,0){$\cC$}
 		\rput[16.0,0.5](0,0){$P_{\lambda}$}
 		\end{pspicture*}
 	}
 	\caption{$\cC \subset P_{\lambda} $.}
 	\label{fig:coincidence_set_in_paraboloid}
 \end{figure}
\\
 \textbf{Step 3.} \emph{Sliding Method.}\\
 We are going to slide the comparison paraboloid in the $e^N$-direction
 until the constant term in the
 expansion matches and show that for that particular $\lambda$, the two solutions coincide.

 For $\lambda < \lambda_0$ we define
 \begin{align}
 	c^\lambda := c_P   +   \lambda \ell(e^N).
 \end{align}
 Then
 \begin{align}\label{eq:expansion_of_translated_and_given_solution}
  	u_{P_\lambda}(x) &= p(x') +\ell(x) +V_{P_\lambda}(x) +c^\lambda \text{ and} \\
 	u(x) &= p(x') + \ell(x) + V_\cC(x) +c \text{ in } \R^N.
 \end{align}
 Observe that we proved in Step 2 that for each 
$\lambda \leq \lambda_0$ such that $c^\lambda < c$,

\begin{align} \label{eq:given_solution_and_lambda_solution_are_compared}
	u_{P_\lambda} \leq u \quad \text{ in } \R^N \text{ and}
\end{align}
	$\cC \subset P_\lambda$.
It follows that
 \begin{align}
 	V_\cC \leq V_{P_\lambda} \quad \text{ in } \R^N.
 \end{align}
 Inserting this into \eqref{eq:expansion_of_translated_and_given_solution} and using \eqref{eq:given_solution_and_lambda_solution_are_compared} we obtain that
 \begin{align}
 	u_{P_\lambda} \leq u = u_{P_\lambda} + V_\cC - V_{P_\lambda} +c -c^\lambda \leq u_{P_\lambda} +c -c^\lambda.
 \end{align}
 Finally, remembering that by \eqref{minustwo}, $c^{\lambda_0}<c$ and noting that
 $\lambda\mapsto c_P+\lambda \ell(e^N)$ is a strictly decreasing\news{,} continuous function and that $ \lambda\mapsto u_{P_\lambda}(x)$ is for each $x$ continuous,
 we let
  $\lambda \searrow \bar{\lambda}= (c_P-c)/(-\ell(e^N))$ and obtain that
 \begin{align}
 	u_{P_{\bar{\lambda}}} \leq u \leq 	u_{P_{\bar{\lambda}}} \quad \text{ in } \R^N.
 \end{align}
 It follows that
 \begin{align}
 	u \equiv u_{P_{\bar{\lambda}}} \quad \text{ as well as } \quad \cC = P_{\bar{\lambda}}.
 \end{align} 
 \textbf{Step 4.} \emph{Identification of the sectional ellipsoids.}\\
 From the construction of $P$ by Theorem \ref{prop:existence_of_paraboloid_solutions} we know that the sectional ellipsoids $E' \subset \R^{N-1}$ of $P$ are (up to scaling and translation) such that
 \begin{align}
 	V'_{E'}(x')= 1 - p(x') \quad \text{ for all } x' \in E'.
 \end{align}
 Setting $v'(x') := p(x') -1 + V'_{E'}(x')$ we conclude from \cite[Theorem II]{CaffarelliKarpShahgolian_Annals_2000} that $v'$ is a nonnegative, global solution of $\Delta v' = \chi_{\{v' >0\}}$ in $\R^{N-1}$. Furthermore the fact that $E' \subset \R^{N-1}$ is bounded implies that $V'_{E'}(x') \to 0$ as $|x'| \to \infty$ and therefore $\frac{v'(\rho x')}{\rho^2} \to p(x')$ in $L^\infty(\partial B_1)$ as $\rho \to \infty$. \\
 This finishes the proof of Main Theorem**.
\qed


\appendix

\section{Applications}

\subsection{Potential Theory and the obstacle problem}\label{sec:Potential_Theory_and_obstacle_problem}
\subsubsection{Ellipsoidal Potential Theory}\label{pot-theory}
In this section we shall give a short historical  remark on the potential  theoretic setting of  the obstacle problem, related to  Newton's   famous  {\it no gravity in the cavity} theorem\footnote{Newton's {\it Principia}, first book Ch. 12, Theorem XXXI.}, which  states that  spherical shells, with uniform distribution of mass, 
do not exert force in the cavity of the body. This result was generalized by P.-S. Laplace  to 
ellipsoidal  homoeoids.\footnote{I.e., a body bounded by  two similar ellipsoids having their axes in the same line. Later James Ivory (1809) gave a beautiful geometric proof of this result.}
This, in particular, means that the Newtonian potential of a homogeneous ellipsoidal  homoeoid is constant in the cavity of the homoeoid. Since the homoeoid can be represented  as $ E_t \setminus E$ where $E$ is the centered ellipsoid and $E_t= tE$ the dilated ellipsoid for  some  $t > 1$, one obtains  that 
$V_{tE} - V_E = \operatorname{const}$ in $E$. Here  $V_M$ stands for the  Newtonian potential of a homogeneous body $M$ (see  Definition \ref{NP}). 
\\
Now arguing as in \cite[proof of Theorem 5.1, page 596]{DiBenedettoFriedman}, we use that since $\Delta V_E \equiv 1$ in $E$, $V_E$ is analytic in $E$. So we may expand it in some neighborhood of the origin $B_\epsilon \subset E$ into a convergent series of of homogeneous polynomials $p_n$, where each $p_n$  is homogeneous of degree $n$, i.e.
\begin{align}
V_E(x) = \sum \limits_{n=0}^\infty p_n(x) \quad \text{ for all } x \in B_\epsilon.
\end{align}
On the other hand the scaling property of the Newton potential (cf. \eqref{eq:scaling_of_Newton_potential}) implies that  for all $x \in B_\epsilon$
\begin{align}
V_{tE}(x) = t^2 V_E\bra{\tfrac{x}{t}} = \sum \limits_{n=0}^\infty t^{2-n} p_n(x).
\end{align}
 Now the assumption that $V_{t E} -V_E \equiv 0$ in $E$ implies that $p_n\equiv 0$ for all $n \in \N \setminus \{0,2\}$\footnote{Note that there is no linear term $p_1$ since we have assumed that the ellipsoid $E$ is centered.}.
Using the analyticity of $V_E$ in $E$, it follows that $V_E$ is a quadratic polynomial inside $E$.

Since paraboloids may be considered to be  limits of a sequences of ellipsoids, and since
the Newtonian potential of a paraboloid in dimension  $N \geq 6$ is well-defined (cf. Lemma \ref{lem:existence_of_newton_potential}) the Newtonian potential of a limit set is still a quadratic polynomial inside the paraboloid.
Similarly 
cylindrical domains with ellipsoids or paraboloids as base  will have the same property, as long as their Newtonian potential is defined.\footnote{In general one may consider the generalized Newtonian potential  of any domain in all dimensions, see \cite{GNP1}.}


\subsubsection{From potential theory to obstacle problem}\label{OP-formulation}

Let us now  rephrase   the discussion of potentials in the  previous section 
with no reference to integrability of Newtonian kernels. Suppose the Newtonian potential $V_D (x)$ of a domain $D$ (cf. Definition \ref{NP}) is finite and 
suppose furthermore that for some quadratic polynomial $q(x)$, $V_D   (x) = q(x)$ 
inside the domain $D$. 
In particular, this means that the function $u (x): = q(x) -V_D (x)$ is  a solution of the no-sign obstacle problem  
\begin{equation}\label{eq:obstacle}
\Delta u =  \chi_{\R^N \setminus D} , \quad u = 0 \quad \hbox{ in } D\quad \hbox{and} 
\quad  | u (x)  | \leq C(1 + |x|^2) \quad \hbox{ for $x$ in } \R^N,
\end{equation} 
for some $C < +\infty$.\footnote{Here the quadratic growth of $u$ follows from a Harnack-inequality
	argument for $V_D$ (cf. \cite[Theorem 8.17 and Theorem 8.18]{gilbarg2001elliptic}).} 
By  \cite[Theorem II]{CaffarelliKarpShahgolian_Annals_2000},  $\R^N \setminus D= \{ u > 0\}$, so it is more convenient  to replace equation \eqref{eq:obstacle} by
\begin{equation}\label{eq:obstacle2}
\Delta u =  \chi_{\{u >0 \} }\quad  , \quad
u \geq 0\hbox{ in } \R^N .
\end{equation}

This new formulation makes it  possible to consider limit domains of coincidence sets
of such  solutions of the obstacle problem. 
In particular, taking limit domains of sequences of  ellipsoids, we obtain that: 
\begin{center}
	{\it   Half-spaces,   paraboloids, and   cylinders  with  these bases do occur as  coincidence sets in \eqref{eq:obstacle2}.} 
\end{center}

\subsection{Eshelby's Conjecture and the obstacle problem} \label{sec:Eshelby}

In the following brief outline of Eshelby's conjecture we follow \cite{KANG_Eshelby_review}. For more details on the problem we refer the interested reader to this survey.\\
Let $N \in \{2,3\}$ and let $\Omega \in \R^N$ be a \emph{bounded} Lipschitz domain to be inserted into a homogeneous medium of conductivity $1$ that had a uniform electric field $E=-a$ before the insertion.\\
Let us assume that the constant conductivity of $\Omega$ is $k \neq 1$. Then the insertion of the inclusion $\Omega$ perturbs the uniform electric field and the perturbed electric field is given by $E=-\nabla u$, where the potential $u$ solves the electric polarization problem
\begin{align}
	\nabla \cdot \bra{ (1+(k-1) \chi_\Omega) \nabla u   } = 0  \qquad &\text{ in } \R^N,\\
	u(x) - a\cdot x = O(|x|^{1-N}) \qquad&\text{ as } |x| \to \infty.
\end{align}
\emph{Eshelby's conjecture} says that if the electric field $E$ is uniform inside the insertion $\Omega$ for \emph{any} uniform field $a$, then the inclusion $\Omega$ is of an elliptic or ellipsoidal shape.\\
In \cite{Kang-Milton-2008} H. Kang and G. W. Milton proved that Eshelby's conjecture is equivalent to proving that if a simply connected domain $\Omega \in \R^N$ is such that its Newton potential $V_\Omega$ (cf.  Definition \ref{NP}) is quadratic inside $\Omega$ then $\Omega$ is an ellipsoid. And this, as outlined in section \ref{OP-formulation}, is equivalent to characterizing global solutions with bounded coincidence set in the obstacle problem.

\subsection{Paraboloid solutions as traveling waves in the Hele-Shaw problem}\label{sec:Traveling_waves_Hele-Shaw}
We briefly illustrate the tight relationship between global solutions of the obstacle problem and traveling waves in the Hele-Shaw problem. Let to this end $u$ be a non-negative solution of
\begin{align}
\Delta u = \chi_{\set {u>0}} \quad \text{ in } \R^N
\end{align}
such that $\set {u=0} = P$, where $P$ is a paraboloid opening in the $e^N$--direction (cf. Theorem \ref{prop:existence_of_paraboloid_solutions}).  By  \cite[Theorem 5.1]{PetrosyanShahgholianUraltseva_book}). 
we know that 
\begin{align} \label{eq:derivatives_of_paraboloid_global_solution}
\partial_{NN} u \geq 0 \quad \text{ and } \quad \partial_N u \leq 0 \quad \text{ in } \R^N .
\end{align}
We fix  a speed
$c>0$ in the direction $e^N$ and consider
\begin{align}
p(t,x) := -\partial_N u(x-ct e^N) c.
\end{align}
Note that $p$ is non-negative.
A direct calculation yields that
\begin{align} \label{eq:first_equation_for_p}
\Delta p(t,x) = \partial_t \chi_{\set{u(x-ct e^N) >0}} \quad \text{ in } \R^N
\end{align}
in the sense of distributions. From \eqref{eq:derivatives_of_paraboloid_global_solution} we infer that $\chi_{\set {u(x-ct e^N)>0}} = \chi_{\set {p(t,x)>0}}$ and combining this fact with \eqref{eq:first_equation_for_p} we obtain that
\begin{align}
\Delta p = \partial_t \chi_{\set {p>0}} \quad \text{ in } \R \times \R^N,
\end{align} 
i.e. that $p$ is a traveling wave solution of the Hele-Shaw problem in the sense of distributions.

\section{Blow-downs}

\begin{lem}[Uniqueness of blow-downs] \label{lem:uniqueness_of_blow-downs}
 \label{lem:blow_down_independent_of_sequence_of_rescalings} \mbox{}\\
	Let $u$ be a nonnegative global solution of the obstacle problem, i.e $u \geq 0$ solves (in the sense of distributions)
	\begin{align}
	\Delta u = \chi_{ \set {u >0 }}  \quad \text{ in } \R^N
	\end{align}
	and let us define the sequences of rescalings $(r_k)_{k\in \N}$ such that $r_k \to \infty$ as $k \to \infty$, and $(\rho_n)_{n \in \N}$ such that $\rho_n \to \infty$ as $n \to \infty$ satisfying that
	\begin{align}
		\frac{u(r_k x)}{r_k^2} &\to u_0(x) \text{ in } C_{\loc}^{1,\alpha}(\R^N)  \text{ as } k \to \infty \quad \text{ and } \quad \\
	\frac{u(\rho_n x)}{\rho_n^2} &\to \tilde{u}_0(x) \text{ in } C_{\loc}^{1,\alpha}(\R^N)  \text{ as } n \to \infty.
	\end{align}
	Then $u_0 = \tilde{u}_0$.
\end{lem}

\begin{proof}
From  \cite[Proposition 3.17 (iii)]{PetrosyanShahgholianUraltseva_book} and \cite[Proposition 5.3]{PetrosyanShahgholianUraltseva_book} we know that both $u_0$ and $\tilde{u}_0$ are either half-space solutions or polynomial solutions. \\
In the case that $u_0$ is a half-space solution we infer from \cite[proof of Theorem II, Case 2]{CaffarelliKarpShahgolian_Annals_2000} that $u$ must be a half-space solution.  \\
In the case that $u_0$ is a polynomial solution, $\tilde{u}_0$ also must be a polynomial solution (by the argument above, interchanging $\tilde{u}_0$ and $u_0$.)
Let therefore now $A, B \in \R^{N \times N}$ be two symmetric, positive semidefinite matrices such that $\operatorname{tr}(A) = \operatorname{tr}(B)= \tfrac{1}{2}$ such that
\begin{align}
	u_0(x) = x^T A x \quad \text{ and} \quad \tilde{u}_0(x) = x^T B x\quad \text{ for all  } x \in \R^N.
\end{align}
Note that as a consequence of \cite[Proposition 3.17 (v)]{PetrosyanShahgholianUraltseva_book}, for any $p \in (1,\infty)$ it holds that 
	\begin{align}
\frac{u(r_k x)}{r_k^2} &\to u_0(x) \text{ in } W_{\loc}^{2,p}(\R^N)  \text{ as } k \to \infty \quad \text{ and } \quad \\
\frac{u(\rho_n x)}{\rho_n^2} &\to \tilde{u}_0(x) \text{ in } W_{\loc}^{2,p}(\R^N)  \text{ as } n \to \infty.
\end{align}
	Define now  $\phi(h, r, x)$ to be the ACF-functional 
\begin{align}
\phi(h, r, x) := \frac{1}{r^4} \int \limits_{B_r(x)} \frac{\abs {\nabla h^+}^2}{\abs {y}^{N-2} }  \dx{y} \int \limits_{B_r(x)} \frac{\abs {\nabla h^-}^2}{\abs {y}^{N-2} }  \dx{y},
\end{align}
which is non-decreasing in $r$, see \cite{ACF84}. Then for a subsequence $(\rho_{n_k})_{k \in \N}$ such that $r_k \leq \rho_{n_k}$ for all $k \in \N$ and for any $e \in \partial B_1$ the monotonicity of the ACF-functional together with the strong $W^{2,p}$-convergence implies that
\begin{align}
	\phi(\partial_e u_0,0,1)   \leftarrow \phi\big( \partial_e \tfrac{u(r_{k} \cdot)}{r_{k}^2}, 1, 0 \big) = \phi(\partial_e u, r_k,0) \leq \phi(\partial_e u, \rho_{n_k}, 0 ) = \phi\big( \partial_e \tfrac{u(\rho_{n_k} \cdot)}{\rho_{n_k}^2}, 1, 0 \big) \to \phi(\partial_e \tilde{u}_0, 1,0)
\end{align}
as $k \to \infty$, i.e. for all $e \in \partial B_1$ it holds that $	\phi(\partial_e u_0,0,1) \leq  \phi(\partial_e \tilde{u}_0, 1,0)$. From this we conclude that for all $e \in \partial B_1$
\begin{align}
	|Ae|^2 \leq |Be|^2.
\end{align}
Using \cite[Lemma 14]{Caffarelli-revisited} we obtain that $A=B$. Hence $u_0 \equiv \tilde{u}_0$.
\end{proof}

\begin{lem}[Preservation of blow-down] \label{lem:preservation_of_blow_down}
	\mbox{}\\
	Let $u$ be a nonnegative global solution of the obstacle problem, i.e $u \geq 0$ solves (in the sense of distributions)
	\begin{align}
	\Delta u = \chi_{ \set {u >0 }}  \quad \text{ in } \R^N
	\end{align}
	and let us define the sequence of rescalings for all $k \in \N$
	\begin{align}
	u_k(x) := \frac{u(x^k +r_k x)}{r^2_k} \quad \text{ for all } x \in \R^N,
	\end{align}
	where $(x^k)_{k \in \N} \subset \partial \{u>0\}$ and $(r_k)_{k \in  \N} \subset (0,\infty)$ such that $r_k \to \infty$ as $k \to \infty$. It is well known that (up to taking a subsequence)
	\begin{align}
	u_k \to u_0 \quad \text{ in } C_{\loc}^{1,\alpha}(\R^N) \text{ as } k \to \infty
	\end{align}
	for all $\alpha \in (0,1)$, where $u_0$ is again a nonnegative global solution of the obstacle problem (cf. \cite[Proposition 3.17]{PetrosyanShahgholianUraltseva_book}).\\	
	Let furthermore $p,q$ be two homogeneous polynomials of degree $2$ such that $p$ is the blow-down of $u$ and $q$ is the blow-down of $u_0$, i.e.
	\begin{align}
	\frac{u(\rho x)}{\rho^2} &\to p(x) \text{ in } C_{\loc}^{1,\alpha}(\R^N)  \text{ as } \rho \to \infty \quad \text{ and } \quad \\
	\frac{u_0(\rho x)}{\rho^2} &\to q(x) \text{ in } C_{\loc}^{1,\alpha}(\R^N)  \text{ as } \rho \to \infty.
	\end{align}
	Then $p=q$.
\end{lem}
\begin{proof}
	Define now  $\phi(h, r, x)$ to be the ACF-functional 
	\begin{align}
	\phi(h, r, x) := \frac{1}{r^4} \int \limits_{B_r(x)} \frac{\abs {\nabla h^+}^2}{\abs {y}^{N-2} }  \dx{y} \int \limits_{B_r(x)} \frac{\abs {\nabla h^-}^2}{\abs {y}^{N-2} }  \dx{y},
	\end{align}
	which is non-decreasing in $r$, see \cite{ACF84}. Note that as a consequence of \cite[Proposition 3.17 (v)]{PetrosyanShahgholianUraltseva_book} for any $p \in (1,\infty)$
	\begin{align}
		\frac{u(\rho x)}{\rho^2} &\to p(x) &&\text{ in } W_{\loc}^{2,p}(\R^N)  \text{ as } \rho \to \infty \quad \\
	\frac{u_0(\rho x)}{\rho^2} &\to q(x) &&\text{ in } W_{\loc}^{2,p}(\R^N)  \text{ as } \rho \to \infty \quad \text{and} \\\quad u_k &\to u_0~ &&\text{ in } W^{2,p}_{\loc}(\R^N) \text{ as } k \to \infty.
	\end{align}
	We may therefore estimate for each $e \in \partial B_1$, $\rho >0$ and $\epsilon >0$ and sufficiently large $k \in \N$
	\begin{align}
	\phi(\partial_e u_0, \rho, 0) &\leq \epsilon + \phi(\partial_e u_k, \rho, 0) = \epsilon + \phi(\partial_e u, \rho r_k, x^k) \leq \epsilon + \lim \limits_{\kappa \to \infty}   \phi(\partial_e  u, \kappa, x^k) \\
	&\leq  \epsilon + \lim \limits_{\kappa \to \infty}  \bra {  \frac{\kappa +\abs{x^k}  }{\kappa}   }^4    \phi(\partial_e  u, \kappa + \abs{x^k},0)  = \epsilon + \phi(\partial_e p , 1, 0)
	\end{align}
	On the other hand using the continuity of the ACF-functional and passing to the limit $\rho \to \infty$ we get
	\begin{align}
	\phi(\partial_e q, 1,0 ) = \lim \limits_{\rho \to \infty} 	\phi(\partial_e u_0, \rho, 0) \leq \epsilon + \phi(\partial_e p , 1, 0).
	\end{align}
	Since $\epsilon >0$ is arbitrary we obtain that for all $e \in \partial B_1$
	\begin{align} \label{eq:ACF_comparison_between_blowup_and_blowdown2}
	\phi \bra {\partial_e q,1,0 } \leq	\phi \bra {\partial_e p , 1,0   } .
	\end{align} 
	Let us now express $p$ and $q$ as
	\begin{align}
	p(x) = x^T A x \quad \text{ and } \quad q(x) = x^T Q x,
	\end{align}
	where $A$ and $Q \in \R^{N \times N}$ are symmetric positive semidefinite such that $\operatorname{tr}(A)= \operatorname{tr}(Q) = \frac{1}{2}$.
	From \eqref{eq:ACF_comparison_between_blowup_and_blowdown2} we conclude that for every $e \in \partial B_1$ 
	\begin{align}
	\abs{Q e}^2 \leq \abs{A e}^2.
	\end{align}
	Using \cite[Lemma 14]{Caffarelli-revisited} we obtain that $A=Q$. Hence $q \equiv p$.
\end{proof}

\bibliographystyle{abbrv}
\bibliography{references-4}
\end{document}